\documentclass[10pt]{article}
\usepackage[margin=1in]{geometry}
\usepackage{amsmath,amsthm,amssymb,eurosym}
\usepackage{enumerate}
\usepackage{graphicx}
\usepackage{fancyhdr}
\usepackage{lastpage}
\usepackage{titlesec}
\usepackage[OT2,T1]{fontenc} 
\usepackage[russian,american]{babel} 
\usepackage{lastpage}
\usepackage{float}
\usepackage{bbm}
\usepackage[usenames,dvipsnames]{xcolor}
\usepackage{color}
\usepackage{wrapfig}
\usepackage{verbatim}
\usepackage{algorithm}
\usepackage{algorithmic}
\usepackage{framed} 
\usepackage[utf8]{inputenc}
\usepackage[T1]{fontenc}
\usepackage{csquotes} 

\newcommand{\R}{\mathbb{R}}     

\renewcommand{\P}{\mathbf{P}}   
\newcommand{\E}{\mathbf{E}}     
\newcommand{\var}{\operatorname{Var}}   
\newcommand{\cov}{\operatorname{Cov}}   

\newcommand{\cF}{\mathcal{F}}

\newcommand{\cX}{\mathcal{X}}

\newcommand{\cN}{\mathcal{N}}

\newcommand{\cW}{\mathcal{W}}

\newcommand{\one}{\mathbbm{1}}  

\newcommand{\pto}{\stackrel{p}{\to}}
\newcommand{\wto}{\Rightarrow}

\newcommand{\sumin}{\sum_{i=1}^n}   
\newcommand{\sol}{\begin{proof}[Solution.]}
\newcommand{\esol}{\end{proof}}
\newcommand{\pf}{\begin{proof}}
\newcommand{\epf}{\end{proof}}

\theoremstyle{plain}

\newtheorem{corollary}{Corollary}
\newtheorem{lemma}{Lemma}
\newtheorem{proposition}{Proposition}
\newtheorem{assumption}{Assumption}
\theoremstyle{definition}
\newtheorem{definition}{Definition}

\usepackage{url}
\usepackage{booktabs}
\usepackage{multirow}
\usepackage{subcaption}
\usepackage{hyperref}
\usepackage[utf8]{inputenc}

\usepackage{thmtools}
\usepackage{thm-restate}

\usepackage{natbib}
\setcitestyle{authoryear,aysep={}}

\newcommand\hl[1]{%
  \bgroup
  \hskip0pt\color{red}%
  #1%
  \egroup
}
\renewcommand{\hl}{}

\newcommand{\avgin}{\frac{1}{n}\sumin}
\newcommand{\limn}{\lim_{n\to\infty}}

\newcommand{\wass}{d_{\cW}}
\newcommand{\rhomax}{\rho_{\text{max}}}

\title{Central limit theorems via Stein's method for randomized experiments under interference\thanks{The author thanks David Choi, Fredrik S{\"a}vje, Johan Ugander, and seminar participants at Yale University and Stanford University for helpful comments and suggestions.  This work was supported in part by NSF grant IIS-1657104.}}
\author{
    Alex Chin\thanks{Department of Statistics, Stanford University, Stanford, CA, 94305 USA (\texttt{ajchin@stanford.edu})}
}
\date{This version: \today}

\begin{document}
\maketitle

\begin{abstract}
We study conditions under which treatment effect estimators constructed under the no-interference assumption in randomized experiments are asymptotically normal in the presence of interference.  We prove that the standard Horvitz-Thompson estimator is asymptotically normal under a restricted interference condition characterized by limiting the degree of the dependency graph.  The amount of interference is allowed to grow with the population size.  We then provide a central limit theorem for the difference-in-means estimator that can handle interference that exists between all pairs of units, provided most of the interference is \hl{captured by a restricted-degree dependency graph}.  The asymptotic variance admits a decomposition into two terms: (a) the variance that is expected under no-interference and (b) the additional variance contributed by interference.  \hl{We propose a conservative variance estimator based on this variance decomposition.}  The results arise as an application of Stein's method.  For practitioners, our results show that standard estimators continue to exhibit normality in large sample sizes and that inference can be made robust to mild forms of interference. 
\\[\baselineskip]
\end{abstract}

\section{Introduction}
In randomized experiments it is standard to assume that units do not interfere with each other~\citep{cox1958planning}.  Such an assumption of no-interference is also known as \emph{individualistic treatment response}~\citep{manski2013identification} and is a key part of the \emph{stable unit treatment value assumption} (SUTVA)~\citep{rubin1974estimating,rubin1980randomization}.  However, in many social, medical, and online settings it is common that the no-interference assumption fails to hold~\citep{rubin1990comment,rosenbaum2007interference,hudgens2008toward,walker2014design,aral2016networked,taylor2017randomized}.

There has been a wealth of recent research into methods for handling interference in randomized experiments.  In some cases it is possible to make reasonable structural assumptions about the nature of interference.  The most well-studied such assumption, known as \emph{partial interference}, is the case in which individuals can be partitioned naturally into groups, like households or schools, such that interference may exist arbitrarily between individuals within the same group but not between individuals of different groups~\citep{sobel2006randomized,hudgens2008toward}.  Partial interference is often paired with an additional exchangeability assumption known as \emph{stratified interference}, which assumes that the potential outcomes are only a function of the number of within-group treated individuals and not the identity of those individuals.  In this setting, sizable contributions have been made regarding how best to use two-stage randomized designs or random saturation designs to estimate a variety of direct and indirect effects and, more generally, dose-response curves~\citep{vanderweele2011effect,tchetgen2012causal,liu2014large,baird2016optimal,basse2017exact}.

The case of \emph{general} or \emph{arbitrary interference} is more difficult.  Generally, researchers proceed by proposing a set of \emph{exposure conditions} that inform the interference pattern~\citep{manski2013identification}.  For example, one might assume that interference is local in nature. This idea is the basis for local interference assumptions such as the \emph{neighborhood treatment response} assumption, which assumes that the potential outcomes of unit $i$ are constant conditional on all treatments in a local neighborhood of $i$~\citep{ugander2013graph,eckles2017design}. \citet{aronow2017estimating} provide unbiased estimators and randomization-based inference under the assumption that the exposure model is correctly specified. \citet{choi2017estimation} studies the case where the treatment effects are monotone, and \citet{sussman2017elements} develop unbiased estimators under neighborhood interference response for various parametric assumptions.  In \emph{graph cluster randomization}, the researcher attempts to reduce bias by using a clustered experimental design in which the clusters have been determined using a graph clustering algorithm designed to minimize edge cuts in a suitably chosen graph~\citep{ugander2013graph,eckles2017design,pouget-abadie17,saveski2017detecting}.

In online settings, in which a standard experimental platform has been operationalized, two-stage or clustered designs may be edge use cases and so may require significant engineering effort to set up.  Such experiments also sacrifice statistical power if it turns out the interference was weak or non-existent.  Therefore, it is of interest to practitioners to be able to tell when controlling for interference is necessary, and when it is appropriate to use standard estimators constructed under the no-interference assumption.  This is especially pertinent in the case of general interference, when there may be no clearly observable structures in the data to indicate whether interference is present.  One option is to develop hypothesis tests for testing for spillover or interference effects~\citep{aronow2012general,athey2017exact,basse2017exact}.  Another study that attempts to move the literature in this direction, and the one that is most relevant to the present work, is that of~\citet{savje2017average}.  In that paper, the authors develop a framework for studying the behavior of standard estimators under a weak form of interference.  They characterize interference based on the notion of a \emph{interference dependence graph}, which defines an edge between two units $i$ and $j$ if there exists some unit $k$ (which is possibly $i$ or $j$) whose treatment affects the responses of both $i$ and $j$.  They establish consistency results for various estimators and experimental designs under the restriction that the average degree of the interference dependence graph grows at rate $o(n)$.

Beyond consistency, it is desirable to know whether estimators satisfy a central limit theorem so that valid asymptotic inference can be performed.  This paper makes two contributions toward this goal.  First, we demonstrate that the interference dependence graph of~\citet{savje2017average} is equivalent to the \emph{dependency graph} introduced by~\citet{chen1975poisson}, used in a variant of Stein's method for bounding distances between random variables.  We show that in a Bernoulli randomized experiment, a central limit theorem exists for the Horvitz-Thompson version of the difference-in-means estimator if one is willing to constrain the maximal degree of the dependency graph at rate $o(n^{1/4})$, rather than the average degree at rate $o(n)$.

In practice the dependency graph may be quite dense, and may even have edges between every single pair of units in the population.  As an example of how this may occur, consider the time-dynamic model studied in \citet{eckles2017design} for experiments conducted on a social network.  In this model, similar in spirit to the linear-in-means model of~\citet{manski1993identification} for capturing endogenous social effects, individuals observe the responses of other individuals and use that information to inform their actions in the following time period.  Interference thus spreads through the network over time and, provided the network is connected, eventually creates long-range dependencies between all pairs of nodes.   Therefore any local model of interference, such as the neighborhood treatment response condition, does not apply.  Our second contribution is to propose a notion of \emph{approximate local interference} that allows for long-range dependencies.  We use a more general form of Stein's method to prove a central limit theorem in a setting where all units may interfere with all other units, \hl{but ``strong interference'' is contained to neighborhoods of size $o(n^{1/3})$.} 

We find that the asymptotic variance of the difference-in-means estimator can be decomposed into two pieces: (a) the variance that results from conditioning on the standard potential outcomes $Y_i^{(0)}$ and $Y_i^{(1)}$, which would have been the true variance under SUTVA; and (b) the additional variation of $Y_i^{(0)}$ and $Y_i^{(1)}$ resulting from interference.  If the additional variation due to interference is sufficiently large then standard confidence intervals may be anticonservative.  \hl{The variance decomposition suggests that a conservative variance estimator may be constructed if the additional variation from interference can be estimated; in this paper we show that this is indeed the case under the restricted interference conditions discussed above.  Adding this additional term to a standard SUTVA variance estimator, such as the Neyman conservative variance estimator, yields confidence intervals that are robust to interference.}

Our technical results rely heavily on Stein's method, a flexible family of approaches for bounding distances between functions of random variables.  As such, it can be used for proving central limit theorems when it is difficult or impossible to make stronger assumptions such as independence or the existence of identically distributed random variables.  Stein's method develops from the seminal paper~\citep{stein1972bound}, which provides a bound for the error in the normal approximation of a sum of random variables with a certain dependency structure.  We provide a short summary of the relevant literature here, but for a longer exposition on the historical development of Stein's method we refer the reader to the surveys found in~\citet{ross2011fundamentals} and ~\citet{chatterjee2014short}.

The theory of dependency graphs, as a particular version of Stein's method, was developed in~\citep{chen1975poisson,stein1986approximate,baldi1989normal,chen2004normal}, and is used for establishing limit theorems when dependence is exactly contained within a small neighborhood of variables.  The dependency graph method is also similar in spirit to the idea of $m$-dependence for sequences of random variables; see for example~\citep{hoeffding1948central,berk1973central,romano2000more}. Section 2 of~\citet{chatterjee2014short} summarizes the main idea of the dependency graph approach.

Classical versions of Stein's method have the property that the random variables need to satisfy some ``nice'' condition---in the case of dependency graphs, that the degree is limited.  The papers~\citet{chatterjee2008new,chatterjee2009fluctuations} develop a more general version of Stein's method that~\citet{chatterjee2014short} calls the \emph{generalized perturbative method}.  This approach formalizes the idea that exact independence is really not too different from approximate independence when it comes to establishing limiting distributional results.  Using this technology we are able to show that asymptotic normality still holds when there exists a weak form of long-range dependencies, even if the induced dependency graph is dense.  In short, if units technically share a dependency edge but this dependence is sufficiently weak, then we may view them as essentially independent of each other.

Our results are of primary interest to two audiences.  First, for practitioners, we contribute to a growing characterization in the literature of understanding when interference is a practical concern and when specialized estimators and robust confidence intervals are needed.  Second, for researchers seeking to establish technical results for causal estimators under interference, our work demonstrates how Stein's method can be a useful machinery for handling the complicated dependencies that often appear among statistical objects in interference problems.

The remainder of the paper is organized as follows.  In Section~\ref{sec:setup} we define notation and assumptions.  Section~\ref{sec:depgraph} discusses a central limit theorem framed in the language of dependency graphs, and Section~\ref{sec:approx-interference} provides a central limit theorem that can handle weak, long-range interference.  \hl{In Section~\ref{sec:variance-estimation} we discuss the construction of a convervative variance estimator.}  In Section~\ref{sec:simulations} we provide simulations and Section~\ref{sec:discussion} concludes.  Proofs are provided in the appendix.

\section{Setup}
\label{sec:setup}
We work within the potential outcomes framework, or Rubin causal model~\citep{neyman1923application,rubin1974estimating}.  Consider a population of $n$ units indexed on the set $[n] = \{1, \dots, n\}$ and let $\+W = (W_1, \dots, W_n) \in \{0, 1\}^n$ be a random vector of binary treatments.  For every individual $i$ and realized vector of treatments $\+w \in \{0, 1\}^n$, we posit the existence of a fixed potential outcome $Y_i(\+w)$.  Note that the potential outcomes are functions of the entire treatment vector and not just the treatment of unit $i$.  We make no parametric restrictions on the form of the potential outcomes.  

\hl{It is also helpful to consider an alternative parametrization of the potential outcomes which clarifies the direct and indirect effects.  Let $\+W_{-i}$ denote the vector of $n-1$ elements obtained by removing the $i$-th element from $\+W$, and partition the vector $\+W$ into the \emph{direct} or \emph{ego treatment} $W_i$ and the \emph{indirect treatment} $\+W_{-i}$.  Then we may index the potential outcomes by these two arguments, writing $Y_i(w_i, \+w_{-i})$ instead of $Y_i(\+w)$.

Our estimand of interest will be a version of the direct effect.  In order to formally define this estimand, we first define a random version of the SUTVA potential outcomes.  Let $Y_i^{(0)}$ and $Y_i^{(1)}$ be the random variables defined as follows:
\begin{align}
Y_i^{(0)} &= Y_i^{(0)}(\+W_{-i}) = Y_i(0, \+W_{-i}) \label{eqn:outcome0} \\
Y_i^{(1)} &= Y_i^{(1)}(\+W_{-i}) = Y_i(1, \+W_{-i}) \label{eqn:outcome1}.
\end{align}

For $w = 0, 1$, the quantity $Y_i^{(w)}$ represents the potential outcome under the scenario in which unit $i$ is exposed to the treatment condition $W_i = w$.  We have used this notation because $Y_i^{(w)}$ are the random extension of the SUTVA potential outcomes in the presence of interference, and their values may vary depending on the treatment assignments assigned to the other units.  For the rest of this paper we will suppress the explicit dependence on $\+W_{-i}$, but the reader should keep in mind that randomness in $Y_i^{(w)}$ results entirely from randomness in $\+W_{-i}$.  If the no-interference assumption is true, then conditioning on $W_i$ removes all randomness in $Y_i(\+W)$, and so $Y_i^{(w)}$ are degenerate random variables and hence reduce to the standard (fixed) potential outcomes.  }



The conceptual advantage of viewing the potential outcomes in this way is that we can get a handle on the variation that exists before and after conditioning on the direct effect.  A situation in which such conditioning removes most of the variance can be viewed as a scenario in which ``SUTVA approximately holds,'' even if strictly speaking SUTVA is violated. 

If SUTVA fails to hold, the standard average treatment effect is undefined.  We follow~\cite{savje2017average} and \hl{first define the \emph{assignment-conditional average treatment effect}
\[\tau_{\text{ACATE}}(\+W) = \avgin \left[Y_i(1, \+W_{-i}) - Y_i(0, \+W_{-i})\right]\]
The average effects $\tau_{\text{ACATE}}(\+W)$ are well-defined but uninterpretable and unwieldy; a seperate estimand exists for each assignment vector.  Instead we focus on studying the \emph{expected average treatment effect} (EATE)
\[\tau = \E[\tau_{\text{ACATE}}(\+W)],\]
where the expectation is taken over the experimental design.  The EATE $\tau$ is a natural relaxation of the standard average treatment effect in the sense that they coincide whenever SUTVA holds.  As noted by~\citet{savje2017average}, the EATE is the expected change of changing a random unit's treatment in the current experiment.\footnote{\hl{\citet{savje2017average} also consider another version of a direct effect called the \emph{average distributional shift effect} (ADSE)  In this paper we only consider Bernoulli i.i.d.\ designs, in which case the EATE and the ADSE are the same estimand.  However they are not equal to each other in general.  For a further discussion of estimands under interference, the reader is directed to Section 3 of \citet{savje2017average}.}}}
It may be viewed as an expected direct effect, where the marginalization is taken over the indirect treatment assignments.  

Using definitions \eqref{eqn:outcome0} and \eqref{eqn:outcome1}, we see that $\tau$ can also be written
\begin{equation}
\label{eqn:tau}
\tau = \avgin \E[Y_i^{(1)} - Y_i^{(0)}],
\end{equation} In other words, we marginalize the difference of the random variables $Y_i^{(1)}$ and $Y_i^{(0)}$ both over the finite population of $n$ units and over any randomness that is contributed by the indirect effect.  

Regardless of whether or not the no-interference assumption holds, one of $Y_i^{(0)}$ and $Y_i^{(1)}$ is still unobserved.  Let $Y_i = Y_i(\+W) = W_iY_i^{(1)} + (1 - W_i)Y_i^{(0)}$ denote the observed outcome.  Let $N_1 = \sumin W_i$ and $N_0 = \sumin (1 - W_i)$ denote the within-group sample sizes.  We study the behavior of the difference-in-means estimator
\begin{equation}
\label{eqn:tau-hat}
\hat \tau = \frac{1}{N_1} \sumin W_iY_i - \frac{1}{N_0} \sumin (1 - W_i)Y_i.
\end{equation}
Note that this estimator is well-defined even when SUTVA is violated, as $Y_i$ is simply the observed outcome.  \citet{savje2017average} study a wider class of estimators, namely the design-based Horvitz-Thompson and H\'ajek estimators that are typically used when $\P(W_i = 1)$ varies with $i$~\citep{horvitz1952generalization,hajek1971comment}.  The difference-in-means estimator is a special case of the H\'ajek estimator, coinciding when the assignment probabilities are the same for every unit.  For simplicity of exposition our analysis focuses on the difference-in-means estimator and an experimental design in which the assignment probabilities are constant across units.  We briefly discuss generalizations at the end of this paper.

In order to obtain asymptotic results, we follow the standard finite population regime~\citep{freedman2008regressionA,freedman2008regressionB,lin2013agnostic,abadie2017sampling,savje2017average} in which we have access to a sequence of finite populations indexed by size $n$.  Each population is comprised of its own treatments and outcomes and $W_i$ and $Y_i$ now represent triangular arrays of random variables.  We shall largely keep the indexing on $n$ implicit to avoid notational clutter, except when charification is helpful, such as for sequences of dependency graphs.  The only randomness within each population is induced by the treatment assignment vector $\+W$. 
Our goal is to study the limiting behavior of the sequence $\hat \tau - \tau$, subject to appropriate scaling.

Throughout this paper we will make use of the following regularity conditions.  The first two conditions, probabilistic assignment and uniformly bounded fourth moments, are standard regularity conditions for asymptotic analysis of regression estimators of treatment effects.  
\hl{\begin{assumption}[Bernoulli design and probabilistic assignment]
\label{asm:design}
$\P(W_i = 1) = \pi$ independently, where the treatment proportion $\pi$ is bounded away from $0$ and $1$.
\end{assumption}}
Assumption~\ref{asm:design} can be relaxed to allow different assignment probabilities per unit, $\P(W_i = 1) = \pi_i$, at the cost of more complicated calculations.
\begin{assumption}[Bounded fourth moments]
\label{asm:bounded-moments}
$\E[|Y_i|^k]$ are uniformly bounded by a constant for all $i, n$ and all $k \leq 4$.
\end{assumption}
\hl{Because we work only with Bernoulli randomized experiments as specified by Assumption~\ref{asm:design}, Assumption~\ref{asm:bounded-moments} is equivalent to a uniform bound placed on the potential outcomes $|Y_i(\+w)^k|$ for all $i, n, \+w$ and $k \leq 4$.  This equivalence does not hold for arbitrary designs, and we state Assumption~\ref{asm:bounded-moments} in the manner above so as to mimic the form of \citet{savje2017average}'s Assumption 1B.}

We also assume existence of the following limits of the potential outcome moments.  
\begin{assumption}[Existence of limits]
\label{asm:variance-limits}
Let $\bar Y^{(1)} = n^{-1} \sumin Y_i^{(1)}$ and $\bar Y^{(0)} = n^{-1} \sumin Y_i^{(0)}$.  The following limits exist:
\begin{align}
\sigma_1^2 &:= \limn \E\left[\avgin (Y_i^{(1)} - \bar Y^{(1)})^2\right] \nonumber\\
\sigma_0^2 &:= \limn \E\left[\avgin (Y_i^{(0)} - \bar Y^{(0)})^2\right] \nonumber\\
\sigma_{01} &:= \limn \E\left[\avgin(Y_i^{(1)} - \bar Y^{(1)})(Y_i^{(0)} - \bar Y^{(0)})\right] \nonumber \\
\sigma_\tau^2 &:= \limn n\var\left[\bar Y^{(1)} - \bar Y^{(0)}\right].\label{eqn:sigma-tau}
\end{align}
\end{assumption}
The quantities inside the expectation are ``population'' quantities in the sense that they do not involve the treatment assignment.  Because of interference they may be random, which is why the expectation is needed.  A consequential implication of the assumption that $\sigma_\tau^2$ exists (equation~\eqref{eqn:sigma-tau}) is that the population difference of means $\bar Y^{(1)} - \bar Y^{(0)}$ is consistent at a $n^{1/2}$ rate of convergence.  It is possible that limiting results are still achievable even when the difference of means converges at a slower rate, but we do not address this case in this paper.

\section{A dependency graph central limit theorem}
\label{sec:depgraph}
In order for central limit theorems to exist for $\hat \tau - \tau$, the observed outcomes $Y_i$ need to be ``sufficiently independent.''  One way to enforce this constraint is to directly require that enough pairs of units are completely independent.  This idea is formalized via the following definition. 
\begin{definition}
\label{def:dep-graph}
Let $\{X_i\}_{i=1}^n$ be a collection of random variables on the nodes $[n]$ of a graph $D$.  Then $D$ is a \emph{dependency graph} if for any two disjoint sets of nodes $A, B \subset [n]$ such that no edge in $D$ crosses between $A$ and $B$, the sets $\{X_i\}_{i \in A}$ and $\{X_i\}_{i \in B}$ are independent.
\end{definition}
The method of dependency graphs is a classical way of characterizing dependence in collections of random variables; see for example~\citet{baldi1989normal}.  Dependency graphs are not necessarily unique; the complete graph always satisfies Definition~\ref{def:dep-graph}, for example.  In this paper we work with the dependency graph that is minimal in the sense that it has the fewest number of edges satisfying the definition.  

In order to characterize interference between units, we consider dependency graphs on the collection of observed outcomes.  Let $D$ denote the dependency graph on the set of random variables $\{Y_i\}_{i=1}^n$.  \hl{In this case we see that the minimal dependency graph corresponds exactly to the notion of interference dependence considered in~\citet{savje2017average}, via the edge definition in Definition 5 of that paper.  They define the interference dependence graph to have edges
\[D_{ij} = \begin{cases}
1 &\text{if } I_{\ell i}I_{\ell j} \text{ for some } \ell \in [n], \\
0 &\text{otherwise,}
\end{cases}\]
where 
\[I_{ij} = \begin{cases}
1 &\text{if } Y_j(\+w) \neq Y_j(\+w')\text{ for some } \+w, \+w' \text{ such that } \+w_{-i} = \+w_{-i}', \\
1 &\text{if } i = j, \\
0 &\text{otherwise.}
\end{cases}
\]
In other words, because the outcomes are defined as functions of the treatment vector, units $i$ and $j$ are connected in this minimal dependency graph if and only if (a) the treatment of $i$ affects the response of $j$, (b) the treatment of $j$ affects the response of $i$, or (c) the responses of both $i$ and $j$ are affected by the treatment of some third unit.  

Importantly, the dependency graph is \emph{not} the same as the social network or other network structure in which the units may be posited to interact.  The dependency graph simply captures the structure of interference and does not specify the source of that interference.  Furthermore, if the interference is actually induced by a social network $G$, then the dependency graph also depends on the process giving rise to interference.  For example, if the outcome-generating process is such that only neighboring units interfere with each other, then $D$ does have the same edge structure as $G$.  But if interference results from a contagion process spreading over the entire network, then $D$ may be fully-connected even if $G$ is sparse.  Throughout this paper, we use $D$ to denote the dependency graph and reserve $G$ to denote a social network, when we need to refer to it.}

Given a dependency graph defined on a collection of random variables, we can take advantage of bounds from the literature on Stein's method.  Such bounds characterize the Wasserstein distance between sums of random variables and a Gaussian random variable.  Recall that the Wasserstein metric between probability measures $\mu$ and $\nu$ is
\[\wass(\mu, \nu) = \sup\left\{ \left|\int h(x) d\mu(x) - \int h(x) d\nu(x)\right| : h \text{ is } 1\text{-Lipschitz}\right\},\]
where a function $h$ is $1$-Lipschitz if it satisfies $|h(x) - h(y)| \leq |x - y|$.
In this paper we are concerned only with controlling the Wasserstein distance between $\mu$ and a standard 
Gaussian random variable.  For any random variable $S$, denote the distance to Gaussianity as
\[\wass(S) = d_{\cW}(\mu, \nu),\]
where $\mu$ is the law of $S$ and $\nu$ is the law of a standard Gaussian random variable, having density
\[\frac{1}{\sqrt{2\pi}} e^{-x^2/2}.\]
We rely on the following dependency graph bound, which we state as a lemma.
\begin{lemma}[\cite{ross2011fundamentals}, Theorem 3.6]
\label{lem:dep-graph-bound}
Let $X_1, \dots, X_n$ be a collection of random variables such that $\E[X_i^4] < \infty$ and $\E[X_i] = 0$.  Let $\sigma^2 = \var(\sum_i X_i)$ and $S = \sum_i X_i$.  Let $d$ be the maximal degree of the dependency graph of $(X_1, \dots, X_n)$.  Then for constants $C_1$ and $C_2$ which do not depend on $n$, $d$ or $\sigma^2$,
\begin{equation}
\label{eqn:dep-graph-bound}
\wass(S / \sigma) \leq C_1\frac{d^{3/2}}{\sigma^2} \bigg(\sumin \E[X_i^4]\bigg)^{1/2} + C_2 \frac{d^2}{\sigma^3} \sumin \E|X_i|^3.
\end{equation}
\end{lemma}
From here, we see that one can define an appropriate choice of $X_i$ such that $S$ is the desired treatment effect estimator, and then provide conditions so that the right-hand side converges to zero.  However, a caveat is that the summand in the difference-in-means estimator $\hat \tau$ depends on the random sample sizes $N_0$ and $N_1$, which  depend on the treatment assignments of all $n$ units.
Therefore, the dependency graph on $\{X_i\}_{i=1}^n$ unfortunately is complete (fully connected), even if the dependency graph on $\{Y_{i}\}_{i=1}^n$ is not, and Lemma~\ref{lem:dep-graph-bound} is not applicable.

As a result, in this section we restrict ourselves to studying a modified form of the difference-in-means estimator, defined by
\begin{equation}
\label{eqn:ht}
\tilde \tau = \sumin \left[\frac{W_{i}Y_{i}}{n\pi} - \frac{(1 - W_{i})Y_{i}}{n(1 - \pi)}\right].
\end{equation}
The estimator $\tilde \tau$ is a Horvitz-Thompson~\citep{horvitz1952generalization} variant of the difference-in-means estimator $\hat \tau$, and uses the population sample sizes $n\pi$ and $n(1 - \pi)$ in the denominator in place of the empirical sample sizes $N_1$ and $N_0$.  Though there is little advantage to using $\tilde \tau$ over $\hat \tau$ in practice, it is still instructive for seeing how the dependency graph method works.  Results are provided for the difference-in-means estimator $\hat \tau$ in Section~\ref{sec:approx-interference}.

We now define a limited interference condition that constrains the structure of the dependency graph.  The metric that we use to measure the extent of interference for a collection of random variables is the maximal degree of the dependency graph.  Let $D_n$ denote the sequence of dependency graphs and $d_n$ denote the corresponding maximal degrees.  We make the following interference assumption about the limiting behavior of $d_n$.  
\begin{assumption}[Local interference]
\label{asm:max-deg}
$d_n = o(n^{1/4})$.
\end{assumption}
This assumption is a local interference assumption in the sense that it requires all interference for a given unit to come from a small number of other units.  \hl{This assumption would hold, if for example, units on a social network $G$ only interfered with neighboring units, and $G$ itself had maximal degree $o(n^{1/4})$.}  For comparison, consider the restricted interference assumption~\citep[Assumption~2] {savje2017average}, which requires the average degree of the dependency graph to be of order $o(n)$.  Our Assumption~\ref{asm:max-deg} is stronger, but it still allows the amount of interference to grow with $n$.  By restricting the maximal degree rather than the average degree, we can apply Lemma~\ref{lem:dep-graph-bound}.

Under the notion of local dependence defined in Assumption~\ref{asm:max-deg}, we obtain the following asymptotic normality result for the Horvitz-Thompson estimator.
\begin{restatable}{theorem}{thmdepgraph}
\label{thm:clt-depgraph}
Let $\tau$ and $\tilde \tau$ be defined as in equations~\eqref{eqn:tau} and \eqref{eqn:ht}.  Under regularity conditions (Assumptions \ref{asm:design}-\ref{asm:variance-limits}) and the restricted dependency degree condition (Assumption~\ref{asm:max-deg}), $\sqrt{n}(\tilde \tau - \tau)$ is asymptotically Gaussian:
\[\sqrt{n}(\tilde \tau - \tau) \wto N(0, \sigma^2),\]
where 
\[\sigma^2 = \limn n \var(\tilde \tau).\]
\end{restatable}
We arrive at Theorem~\ref{thm:clt-depgraph} by defining an appropriate choice for $X_i$ and evaluating the variance $\sigma^2$, which allows us to control the bound in Lemma~\ref{lem:dep-graph-bound}.  The full proof is provided in the appendix.

A curious feature of Stein's method is that it allows one to make statements about the asymptotic behavior of random objects without calculating an explicit expression for the variance.  Because our primary interest is not in the Horvitz-Thompson estimator $\tilde \tau$, we skip calculating the limiting variance $\sigma^2$,
but is not hard to express it in terms of the moments defined in Assumption~\ref{asm:variance-limits}.  For the difference-in-means estimator in Section~\ref{sec:approx-interference} we provide an explicit characterization of the limiting variance.

\section{A central limit theorem for approximate local interference}
\label{sec:approx-interference}

There are two drawbacks of relying on the dependency graph approach for studying treatment effect estimators.  It does not allow us to study estimators like the difference-in-means estimator that use empirical sample sizes, and it requires exact local interference (Assumption~\ref{asm:max-deg}).  In this section we discuss how these issues can be overcome.  Rather than require most pairs of nodes to be exactly independent, we only require approximate independence, which allows long-range interference as long as it is not too strong.

\hl{It is worth explaining how this approximate independence assumption might arise in practice.  Suppose we measure a social network, $G$, the edges of which capture the peer interactions which we believe transmit the interference mechanism.  If we believe the neighborhood exposure assumptions invoked by, e.g., \citet{ugander2013graph,forastiere2016identification,sussman2017elements,jagadeesan2017designs}, and others, then the dependency graph methods of Section~\ref{sec:depgraph} are sufficient.  However, exact local interference is insufficient for describing more complex data generating processes.  A social contagion process, such as that considered by~\cite{eckles2017design}, leads to a fully-connected dependency graph $D$ even if the social graph $G$ is sparse (but connected).  This discrepancy between $D$ and $G$ was previously discussed in Section~\ref{sec:depgraph} and may be discouraging to practitioners.

However, if peer effects dissipate over the network, we may believe that interference from long-distance units in $G$ may be second- or lower-order effects.
One may conceptualize the existence of two different dependency graphs defined on the collection of units, one capturing strong interference and the other capturing weak interference.  Here we provide a result that allows the weak interference graph to be arbitrarily dense, but requires $o(n^{1/3})$ sparsity in the strong interference graph.  In this case, sparsity of the social graph $G$ would be sufficient for the limiting results to hold.  Such sparsity has been demonstrated empirically on such social networks as Facebook~\citep{ugander2011anatomy}, MSN~\citep{leskovec2008planetary}, and a mobile phone newtork~\citep{onnela2007structure}, and suggested theoretically by Dunbar's number, a suggested sociocognitive limit in the number of possible stable social relationships~\citep{dunbar1992neocortex}.

The primary assumption is provided in Assumption~\ref{asm:approx-local}, but we require a technical detour to describe the main approach, developed by~\citet{chatterjee2008new}.}  Let $\+X = (X_1, \dots, X_n) \in \cX^n$ be a random vector of independent random variables on a measure space $\cX$ and let $f: \cX^n \to \R$ be a scalar-valued measurable function.  The objective is to bound the distance to normality of $S:= f(\+X)$.  To do so, we characterize the behavior of $f$ when $\+X$ is ``perturbed'' by replacing some components of $\+X$ by independent copies.  Let $\+X' = (X_1', \dots, X_n')$ denote an independent copy of $\+X$.  For every $A\in [n]$ let $\+X^A$ be the vector where the entries corresponding to $A$ are replaced by elements of $\+X'$, defined componentwise as
\[X_i^A = \begin{cases}
X_i' &\text{if }i \in A \\
X_i &\text{if }i \not \in A
\end{cases}.\]
Now define
\begin{align*}
\Delta_i f = f(\+X) - f(\+X^i), &\qquad i \in [n], \\
\Delta_i f^A = f(\+X^A) - f(\+X^{A \cup i}), &\qquad A \subset [n], i \not \in A,
\end{align*}
where we have made a notational simplification by writing $\+X^i$ instead of $\+X^{\{i\}}$ and $\+X^{A \cup i}$ instead of $\+X^{A \cup \{i\}}$.
The quantities $\Delta_i f$ and $\Delta_i f^A$ can be viewed as discrete derivatives, because they measure the change in the function $f$ in response to perturbations of $\+X$.  If perturbations in $\+X$ act upon $f$ mostly independently by coordinate, then we expect the resulting value $f(\+X)$ to be approximately normal.  We can now state the following normal approximation theorem, which is the main result in~\cite{chatterjee2008new}.  We state it as a lemma.
\begin{lemma}[\cite{chatterjee2008new}, Theorem 2.2]
Let $\+X = (X_1, \dots, X_n)$ be a vector of independent real-valued random variables, and let $S = f(\+X)$.  Suppose $\E[S] = 0$ and $\sigma^2 := \E[S^2] < \infty$.  Define
\[T = \frac{1}{2} \sum_{i=1}^n \sum_{A \subset [n] \setminus \{i\}} \frac{\Delta_i f \Delta_i f^A}{n\binom{n-1}{|A|}} \]
Then $\E T = \sigma^2$ and
\[
\wass(S/\sigma) \leq \frac{1}{\sigma^2}[\var(\E(T | S))]^{1/2} + \frac{1}{2\sigma^3} \sumin \E[|\Delta_i f|^3].
\]
\label{lem:chatterjee}
\end{lemma}
It is more convenient to study a version of Lemma~\ref{lem:chatterjee} that characterizes the Wasserstein distance in terms of local dependencies.  This corollary is essentially a variant of Corollary 3.2 in~\citet{chatterjee2014short}.
\begin{restatable}{corollary}{corchatterjee}
\label{cor:chatterjee}
Let all variables be defined as in Lemma~\ref{lem:chatterjee}.  For every $i, j$, let $c_{i,j}$ be a constant such that for all $A \in [n] \setminus \{i\}$ and $B \in [n] \setminus \{j\}$,
\[\cov(\Delta_i f \Delta_i f^A, \Delta_j f \Delta_j f^B) \leq c_{i,j}.\]
Then
\begin{equation}
\label{eqn:wass-perturbative}
\wass(S/\sigma) \leq \frac{1}{2\sigma^2} \bigg(\sum_{i,j=1}^n c_{i,j}\bigg)^{1/2} + \frac{1}{2\sigma^3} \sumin \E[|\Delta_i f|^3].
\end{equation}
\end{restatable}

One may gain intuition for Corollary~\ref{cor:chatterjee} by considering the case of a dependency graph, in which an upper bound can be provided for the number of covariance terms $c_{i,j}$ that can be nonzero.  \hl{Let $D$ be a graph with maximal degree $d_n$ and let $\cN_i$ denote the neighborhood of unit $i$.  Letting $\+X$ be a collection of independent random variables, as in Lemma~\ref{lem:chatterjee}, consider a function of the form $S = f(\+X) = \sum_{i=1}^n g_i(X_i, \+X_{\cN_i})$, where $\+X_{\cN_i}$ are the elements of $\+X$ restricted to $\cN_i$.  Notice now that $D$ is a dependency graph for the collection of variables $\{g_i(X_i, \+X_{\cN_i})\}_{i=1}^n$.  Then the discrete derivative has the form $\Delta_i f = \sum_{r \in \cN_i} \Delta_i X_r$, where $\Delta_i X_r = X_r - X_r^i$ is the effect on unit $r$ of perturbing unit $i$.  Notice that the sum in the discrete derivative is only over the $r$ units in $\cN_i$, because the only arguments of $g_i$ are elements of $\cN_i$.  Now consider the covariance between the discrete derivatives for units $i$ and $j$, which can be calculated as
\[\cov(\Delta_i f, \Delta_j f) = \sum_{r \in \cN_i} \sum_{s \in \cN_j} \cov(\Delta_i X_r, \Delta_j X_s) = \sum_{r \in \cN_i \cap \cN_j} \cov(\Delta_i X_r, \Delta_j X_r) \leq Cd_n \one(|\cN_i \cap \cN_j| > 0),\]
where $C$ is a constant that does not depend on $n$ or $d_n$.  

In other words, the covariance is always of order $d_n$, but is exactly zero whenever the neighborhoods of $i$ and $j$ do not intersect.  Now, for every unit $i$, the number of units $j$ such that $|\cN_i \cap \cN_j| > 0$ is at most $d_n^2$.  Therefore the total number of covariances that can be nonzero is $nd_n^2$, and so the total magnitude of those covariances is $Cnd_n^3$.  Assuming the variance $\sigma^2$ is of order $n$, and note that $\Delta_i f^A$ is of order at most $d_n$.  Then the right-hand side of equation~\ref{eqn:wass-perturbative} is of the form
\[\frac{C_1}{n} (nd_n^3)^{1/2} + \frac{C_2}{n^{3/2}} nd_n^3 = C_1 \frac{d_n^{3/2}}{n^{1/2}} + C_2 \frac{d_n^3}{n^{1/2}}.\]
This quantity can be made small in the limit if $d_n$ grows sufficiently slowly.  The first term here and the first term in the dependency graph bound~\eqref{eqn:dep-graph-bound} both require a $d_n = o(n^{1/3})$ constraint; the second term here requires a $d_n = o(n^{1/6})$ constraint whereas the second term of equation~\eqref{eqn:dep-graph-bound} requires $d_n = o(n^{1/4})$.}


We return now to the problem of obtaining a limiting result for $\hat \tau$.  Define the sequence of functions
\[f_n(\+W) = \sqrt{n} (\hat \tau - \tau) = \sqrt{n} \sumin \left[\frac{W_{i}}{N_1} - \frac{1 - W_{i}}{N_0}\right]Y_{i}(\+W) - \sqrt{n} \tau.\]
Since the treatment vector $\+W$ is comprised of independent Bernoulli$(\pi)$ random variables and is the sole source of randomness in $f_n$, Corollary~\ref{cor:chatterjee} is applicable provided we define appropriate constraints on the behavior of $f_n$ under perturbations of the treatment vector.

Let $W_{i}'$ denote an independent copy of $W_{i}$ and let $\+W^i$ the resulting treatment vector obtained by swapping out $W_{i}$ for $W_{i}'$ in $\+W$, defined componentwise as
\[W_{j}^i = \begin{cases} 
W_{j} &\text{if } j \neq i \\
W_{j}' &\text{if } j = i
\end{cases}.
\]
Let 
\[Y_{j}^i = Y_{r}(\+W^i)\]
denote the resulting response of unit $r$ when $i$ is perturbed, and define
\[\Delta_i Y_{r} = Y_{r} - Y_{r}^i\]
to be the change in $Y_{r}$ when $W_i$ is replaced with an independent copy.  Furthermore, let 
\begin{align*}
N_1' &= N_1 + W_{i}' - W_{i} \\
N_0' &= n - N_1'
\end{align*}
denote the adjusted sample sizes.

The following lemma, which we state without proof, characterizes the discrete derivative $\Delta_i f_n$.
\begin{lemma}
\label{lem:form-of-derivative}
Let $W_{i}'$, $N_1'$, $N_0'$, and $Y_{r}^i$ be defined as above.
Then for every $i \in [n]$, the discrete derivative can be written as
\begin{equation}
\label{eqn:discrete-deriv}
\Delta_i f_n = \sqrt{n} \bigg(A_{i} + \sum_{r \neq i} B_{i,r}\bigg),
\end{equation}
where
\begin{align*}
A_{i} 
&= \left[\frac{W_i}{N_1} - \frac{W_i'}{N_1'}\right]Y_i^{(1)} - \left[\frac{1 - W_i}{N_0} - \frac{1 - W_i'}{N_0'}\right]Y_i^{(0)} \\
B_{i,r} 
&= \left[\frac{W_r}{N_1}Y_r - \frac{W_r}{N_1'}Y_r^i\right] - \left[\frac{1 - W_r}{N_0} Y_r - \frac{1 - W_r}{N_0'}Y_r^i\right].
\end{align*}
\end{lemma}
Notice that $A_i$ describes the change for unit $i$ (the direct effect) and $B_{i,r}$ describes the effect that perturbing the treatment of unit $i$ has on the response of unit $r$.

\hl{We now describe assumptions that constrain the behavior of $\Delta_i Y_r$, which in turn allows us to handle $\Delta_i f_n$ via the expression in Lemma~\ref{lem:form-of-derivative}.  Assumption~\ref{asm:approx-regularity} provides a set of weak regularity conditions; the main conceptual condition of approximate local interference is provided by Assumption~\ref{asm:approx-local}. 

\begin{assumption}[Covariance regularity conditions]
\label{asm:approx-regularity}
The following global covariance constraints hold:
\begin{enumerate}[(a)]
\item
\[\sumin \sum_{j=1}^n| \cov(Y_i, Y_j)| = o(n^2).\]
\item 
\[\sumin \sum_{r \neq i} \sum_{j \neq i} |\cov(\Delta_i Y_r, Y_j)| = o(n^2).\]
\item 
\[\sumin \sum_{j=1}^n \sum_{r \neq i} \sum_{\substack{s \neq j \\ s \neq r}} |\cov(\Delta_i Y_r, \Delta_j Y_s)| = o(n^2).\]
\end{enumerate}
\end{assumption}
Part (a) states that the overall responses are not too dependent.  Part (b) states that the effect of perturbing $i$ on $r$ is mostly independent of the behavior of unit $j$.  Part (c) states that the effect of perturbing $i$ on $r$ and the effect of perturbing $j$ on $s$ are mostly independent, when $r$ and $s$ are distinct from $i$ and $j$ and each other.  To see why these assumptions may be reasonable, consider the case where SUTVA holds.  Then, the expressions in part (b) and (c) are also exactly zero.  Finally, $\cov(Y_i, Y_j) = 0$ whenever $i \neq j$ so that the expression in part (a) is $O(n)$.

Now we describe the approximate local interference condition.  For every $n$, let $H_n$ denote a undirected graph on the vertices $[n]$ which represents a dependency graph for ``strong interference,'' with weak interference allowed outside of $H_n$.  Let the neighborhood of unit $i$ on $H_n$ be denoted by $\cN_i^{H_n} = \{j \in [n]: (H_n)_{ij} = 1\}$.  Assumption~\ref{asm:approx-local} formally describes the conditions required of $H_n$.  

\begin{assumption}[Approximate local interference]
\label{asm:approx-local}
There exists a sequence of graphs $\{H_n\}_{n=1}^\infty$ having maximal degree sequence $h_n = o(n^{1/3})$ such that
\[\max\left\{\sum_{r \notin \cN_i^{H_n}} |\Delta_i Y_{r}|, \sum_{r \notin \cN_i^{H_n}} |\Delta_r Y_{i}|\right\} \to 0\]
almost surely for every node $i \in [n]$.
\end{assumption}}

\hl{Assumption~\ref{asm:approx-local} states that outside of the strong interference neighborhoods $\cN_i^{H_n}$, the interference has a magnitude that is vanishing in the limit.  If the quantities $\sum_{r \notin \cN_i^{H_n}} |\Delta_i Y_{r}|$ and $\sum_{r \notin \cN_i^{H_n}} |\Delta_r Y_{i}|$ are equal to zero exactly, then $H_n$ becomes a sparse dependency graph of the sort discussed in Section~\ref{sec:depgraph}.  Notably, the neighborhood size $h_n$ is allowed to grow at rate $o(n^{1/3})$.  Note that the maximum is taken over two terms; the first describes outcomes that can be affected by the treatment of unit $i$ and the second describes treatments that can affect the outcome of unit $i$.\footnote{\hl{It is easy to generalize $H_n$ to directed graphs, which would capture the notion that $Y_i$ may depend on $W_j$ without $Y_j$ depending on $W_i$ and vice versa.  We define $H_n$ as undirected here only for notational and conceptual simplicity.}}}

We now state the main result.  In comparison to Theorem~\ref{thm:clt-depgraph}, it replaces a restriction on the dependency graph, Assumption~\ref{asm:max-deg}, with the approximate local interference requirements, Assumptions~\ref{asm:approx-regularity} and~\ref{asm:approx-local}.  It also is a statement about the difference-in-means estimator $\hat \tau$ rather than the Horvitz-Thompson estimator $\tilde \tau$.
\begin{restatable}{theorem}{thmperturbative}
\label{thm:clt-perturbative}
Let $\tau$ and $\hat \tau$ be defined as in equations~\eqref{eqn:tau} and \eqref{eqn:tau-hat}, and assume that the regularity conditions (Assumptions \ref{asm:design}-\ref{asm:variance-limits} and~\ref{asm:approx-regularity}) hold.  Assume further that the outcome functions is constrained according to Assumption~\ref{asm:approx-local}.  Then $\sqrt{n}(\hat \tau - \tau)$ is asymptotically Gaussian:
\[\sqrt{n}(\hat \tau - \tau) \wto N(0, \sigma^2).\]
The limiting variance $\sigma^2$ has the form
\begin{equation}
\label{eqn:asymptotic-variance}
\sigma^2 := \limn n\var(\hat \tau) = \frac{1 - \pi}{\pi} \sigma_1^2 + \frac{\pi}{1 - \pi} \sigma_0^2 + 2 \sigma_{01} + \sigma_\tau^2,
\end{equation}
where the quantities $\sigma_1^2$, $\sigma_0^2$, $\sigma_{01}$, and $\sigma_\tau^2$ are defined in Assumption~\ref{asm:variance-limits}, and $\pi = \limn \P(W_i = 1)$ is the limiting treatment proportion.
\end{restatable}

The asymptotic variance takes the form of a variance decomposition based on conditioning on the ``potential outcomes'' $Y_i^{(0)}$ and $Y_i^{(1)}$ \hl{(the random quantities defined in equations~\eqref{eqn:outcome0} and~\eqref{eqn:outcome1}, not the original fixed potential outcomes $Y_i(\+w)$).}  To see this, denote the $\sigma$-field generated by $Y_i^{(0)}$ and $Y_i^{(1)}$ as
\begin{equation}
\label{eqn:sigma-field}
\cF := \{Y_i^{(w)} : i \in [n],  w \in \{0, 1\}\}.
\end{equation}
Then by the law of total variance,
\[\var(\hat \tau) = \E[\var(\hat \tau | \cF)] + \var[\E(\hat \tau | \cF)].\]
This decomposition is evident in the asymptotic variance $\sigma^2$, as
\begin{equation}
\label{eqn:var-decomp1}
\limn n \E[\var(\hat \tau | \cF)] = \frac{1 - \pi}{\pi} \sigma_1^2 + \frac{\pi}{1 - \pi} \sigma_0^2 + 2 \sigma_{01}
\end{equation}
and
\begin{equation}
\label{eqn:var-decomp2}
\limn n\var[\E(\hat \tau | \cF)] = \sigma_\tau^2.
\end{equation}
\hl{(A proof of this decomposition is provided in the supplementary material as a part of the proof of Theorem 2.)}
The first three terms, \eqref{eqn:var-decomp1}, form the standard asymptotic variance of the difference-in-means estimator under no-interference, in which $Y_i^{(0)}$ and $Y_i^{(1)}$ are fixed quantities~\citep{freedman2008regressionA,freedman2008regressionB,lin2013agnostic}.  The last term, \eqref{eqn:var-decomp2}, captures the additional variation of the ``total population'' average treatment effect, which is variation that remains even if we were able to observe $Y_i^{(0)}$ and $Y_i^{(1)}$ for every unit.  This decomposition implies that if $\sigma_\tau^2$ stabilizes to a non-zero value, then confidence intervals constructed under the SUTVA assumption will only provide correct coverage conditional on $\cF$, and will fail to account for the fact that $Y_i^{(0)}$ and $Y_i^{(1)}$ may exhibit additional variation under the experimental design.  \hl{In Section~\ref{sec:variance-estimation} we provide a consistent estimator of $\sigma_\tau^2$ that can be used to adjust SUTVA-based standard errors.}

\hl{
\section{Variance estimation}
\label{sec:variance-estimation}

We can use the variance decomposition to guide the derivation of an appropriate variance estimator for $\hat \tau$.  In order to estimate the SUTVA portion of variance, it is enough to use the standard Neyman conservative variance estimator,
\[\hat V_{\text{SUTVA}}^2 = \frac{S_1^2}{N_1} + \frac{S_0^2}{N_0},\]
where
\[S_1^2 = \frac{1}{N_1 - 1} \sumin W_i(Y_i - \bar Y_1)^2, \qquad S_0^2 = \frac{1}{N_0 - 1} \sumin (1 - W_i)(Y_i - \bar Y_0)^2\]
where the sample variances for units assigned to the control group and treatment group, respectively, and
\[\bar Y_1 = \frac{1}{N_1} \sumin W_i Y_i, \qquad \bar Y_0 = \frac{1}{N_0} \sumin (1 - W_i) Y_i\]
are the sample means.
(Note that $\bar Y_1$ and $\bar Y_0$ are not the same as $\bar Y^{(1)}$ and $\bar Y^{(0)}$.)
It remains to handle $\sigma_\tau^2$, and under the restricted dependence conditions used in this paper, it is possible to derive a consistent estimator of the variance.  Denote the (random) population average treatment effect by
\[T = \bar Y^{(1)} - \bar Y^{(0)},\]
so that $\tau = \E[T]$ and $\sigma_\tau^2 = \var(T) / n$, and
\[\E[T^2] = \E\left[(\bar Y^{(1)})^2 + (\bar Y^{(0)})^2 - 2 \bar Y^{(1)} \bar Y^{(0)}\right].\]
Then we see that the plugin estimator
\[\hat V_\tau^2 = \bar Y_1^2 + \bar Y_0^2 - 2\bar Y_1 \bar Y_0 - \hat \tau^2\]
is consistent for $\var(T) = \E[T^2] - \tau^2$ as long as $\bar Y_1$ and $\bar Y_0$ are consistent for their limiting expectations.  The following proposition shows that this is indeed the case under a $o(n^{1/3})$ (approximate) maximal degree dependency structure.

\begin{proposition}
\label{prop:variance-estimation}
Under regularity conditions (Assumptions \ref{asm:design}-\ref{asm:variance-limits}) and restricted interference (either Assumption~\ref{asm:max-deg} or Assumptions~\ref{asm:approx-regularity}-\ref{asm:approx-local}), $\hat \sigma_\tau^2$ is consistent for $\sigma_\tau^2$. 
\end{proposition}

The following corollary follows immediately, and shows that $\hat V_\tau^2$ can be used as an ``interference adjustment'' to protect the standard SUTVA variance estimator against the forms of interference discussed in this paper.
\begin{corollary}
The variance estimator $\hat V = \hat V_{\text{SUTVA}}^2 + \hat V_\tau^2$ is asymptotically conservative for the variance of $\hat \tau$.
\end{corollary}
This variance estimator can then be used to construct $1-\alpha$ confidence intervals in the usual way,
\[\hat \tau \pm z_{\alpha / 2} \sqrt{\hat V},\]
where $z_{\alpha/2}$ is the appropriate Gaussian quantile.

}
  
\section{Simulations}
\label{sec:simulations}
This section is devoted to two sets of simulations that are designed to illuminate some of the practical implications of our theoretical findings.  The first simulation involves tests of Gaussianity and considers situations in which asymptotic inference may be invalid.  The second simulation involves the variance decomposition provided by equations~\eqref{eqn:var-decomp1} and~\eqref{eqn:var-decomp2} and considers situations in which inference built under the SUTVA assumption may be invalid.
For both simulations we use the same set of networks and generative response model.

In order to replicate as closely as possible the structural characteristics observed in real-world networks, we use five empirical networks from the \texttt{facebook100} dataset, an assortment of complete online friendship networks for one hundred colleges and universities collected from a single-day snapshot of Facebook in September 2005.  A detailed analysis of the social structure of these networks was given in~\citet{traud2012social}.  The five schools used are the California Institute of Technology, Haverford College, Amherst College, Michigan Technological University, and Wake Forest University; the only reason for the selection of these five particular schools was so as to produce a rough stratification of population sizes.  For each school, we use the largest connected component only.  Table~\ref{table:fb100-summary} contains basic summary statistics for the networks used.

\begin{table}[ht]
\centering
\begin{tabular}{lrrrrr}
  \hline
network & nodes & edges & avg.\ degree & avg.\ pairwise dist. & diameter \\ 
  \hline
Caltech & 762 & 16651 & 43.70 & 2.34 &   6 \\ 
  Haverford & 1446 & 59589 & 82.42 & 2.23 &   6 \\ 
  Amherst & 2235 & 90954 & 81.39 & 2.40 &   7 \\ 
  Michigan Tech & 3745 & 81901 & 43.74 & 2.84 &   7 \\ 
  Wake Forest & 5366 & 279186 & 104.06 & 2.51 &   9 \\ 
   \hline
\end{tabular}
\caption{Summary statistics for the five networks used in the simulation.}
\label{table:fb100-summary}
\end{table}

We use a simple response model that allows us to control the amount of dependence exhibited among the observations.  For network $G$ and nodes $i$ and $j$, let $\tilde Z_{\rho, i, j} = 1$ if nodes $i$ and $j$ are exactly $\rho$ units apart in graph $G$, and then define
\[Z_{\rho, i} = \bigg(\sum_j \tilde Z_{\rho, i, j}\bigg)^{-1} \sum_j \tilde Z_{\rho, i, j} W_j\]
to be the proportion of units which are exactly distance $\rho$ from $i$ that receive the treatment.  Then we model the outcome as
\[Y_i^{(w)} = \alpha_i^{(w)} + \sum_{\rho = 1}^{\rhomax} \beta_\rho^{(w)} Z_{\rho, i}\]
for $w = 0, 1$.
The intercept $\alpha_i = (\alpha_i^{(0)}, \alpha_i^{(1)})$ captures a heterogeneous direct effect.  The maximum distance parameter $\rhomax$ is an integer ranging from $0$ to the diameter of the graph.  By $\rhomax = 0$ we mean the summation is omitted entirely, so that $Y_i^{(w)} = \alpha_i^{(w)}$, and there is no spillover effect and hence no interference.  When $\rhomax = 1$, each unit is subject to a direct effect and a spillover effect governed by coefficient $\beta_1^{(w)}$ and the proportion $Z_{1,i}$ of neighbors of $i$ receiving the treatment.  Analogously, higher values of $\rhomax$ admit more distant sources of interference.

We model the coefficient vector as decaying exponentially in the graph distance,
\[\beta_\rho^{(1)} = 2 \gamma^\rho, \qquad \beta_\rho^{(0)} = \gamma^\rho,\]
for a decay parameter $\gamma \in (0, 1)$.  Therefore, each node receives a direct effect $\alpha_i^{(1)} - \alpha_i^{(0)}$ and an indirect effect
\[\sum_{\rho = 1}^{\rhomax} \gamma^\rho Z_{\rho, i}.\]

We control the amount of dependence by varying the parameters $\rhomax$, which explicitly controls the structure of the dependency graph, and $\gamma$, which controls the rate at which spillover effects dissipate as they travel through the network.  

\subsection{Tests of normality}
We first compute normality test statistics for a variety of parameter configurations.  We vary the decay rate $\gamma$, and the maximum dependency distance $\rhomax$.  The parameter values we use are  $\gamma \in \{0.5, 0.9, 0.99\}$, and $\rhomax \in \{2, 6\}$.  The maximum value $\rhomax = 6$ was used because all networks have diameter at least 6.  For every parameter configuration and each network, we generate $10$ instances of the direct effect.  The direct effect values $\alpha_i^{(1)}$ and $\alpha_i^{(0)}$ are sampled from independent exponential distributions with different means; the treatment group has mean $1 / 0.3$ and the control group has mean $2$.   

For each of the 10 instances, we sample 500 draws of the treatment vector as independent Bernoulli$(0.5)$ variables, and compute the outcomes and resulting difference-in-means estimate.  We report the test statistic and $p$-value of the Shapiro-Wilk (SW) test for normality~\citep{shapiro1965analysis}.   This produces 10 $p$-values for each network and parameter configuration, one for each instance of the direct effect.  Note that we use these $p$-values purely for exploratory purposes and do not require nor attempt multiple comparison control.

The results are displayed in Figure~\ref{fig:sw}.  Recall that $p$-values are uniform under the null hypothesis that the difference-in-means statistics are normally distributed, so that configurations in which most of the $p$-values are small may indicate a departure from normality.  The scenarios representing the greatest amount of interference are those in which the indirect effect is allowed to propagate over a long distance ($\rhomax = 6$) and in which the indirect effect does not decay much ($\gamma = 0.99$) as it travels across the network.  We see that the $p$-values are smallest for these configurations.  For all networks, the value of $\gamma$ needs to be quite large in order to cause serious problems; $p$-values appear to be roughly uniform for a dissipation rate of $\gamma = 0.5$ even when $\rhomax = 6$.  This supports the claim that having a sparse dependency graph is not necessary for asymptotic normality, since for these networks, the induced dependency graph when $\rhomax = 6$ is either complete or nearly complete.  Departures from normality also seem to be sensitive to the particular network structure; the Caltech and Michigan Tech networks seem to be quite well-behaved even under the strongest regimes of interference ($\gamma = 0.99$ and $\rhomax = 6$).

\begin{figure}
\centering
   \begin{subfigure}[b]{0.77\textwidth}
   \includegraphics[width=\textwidth]{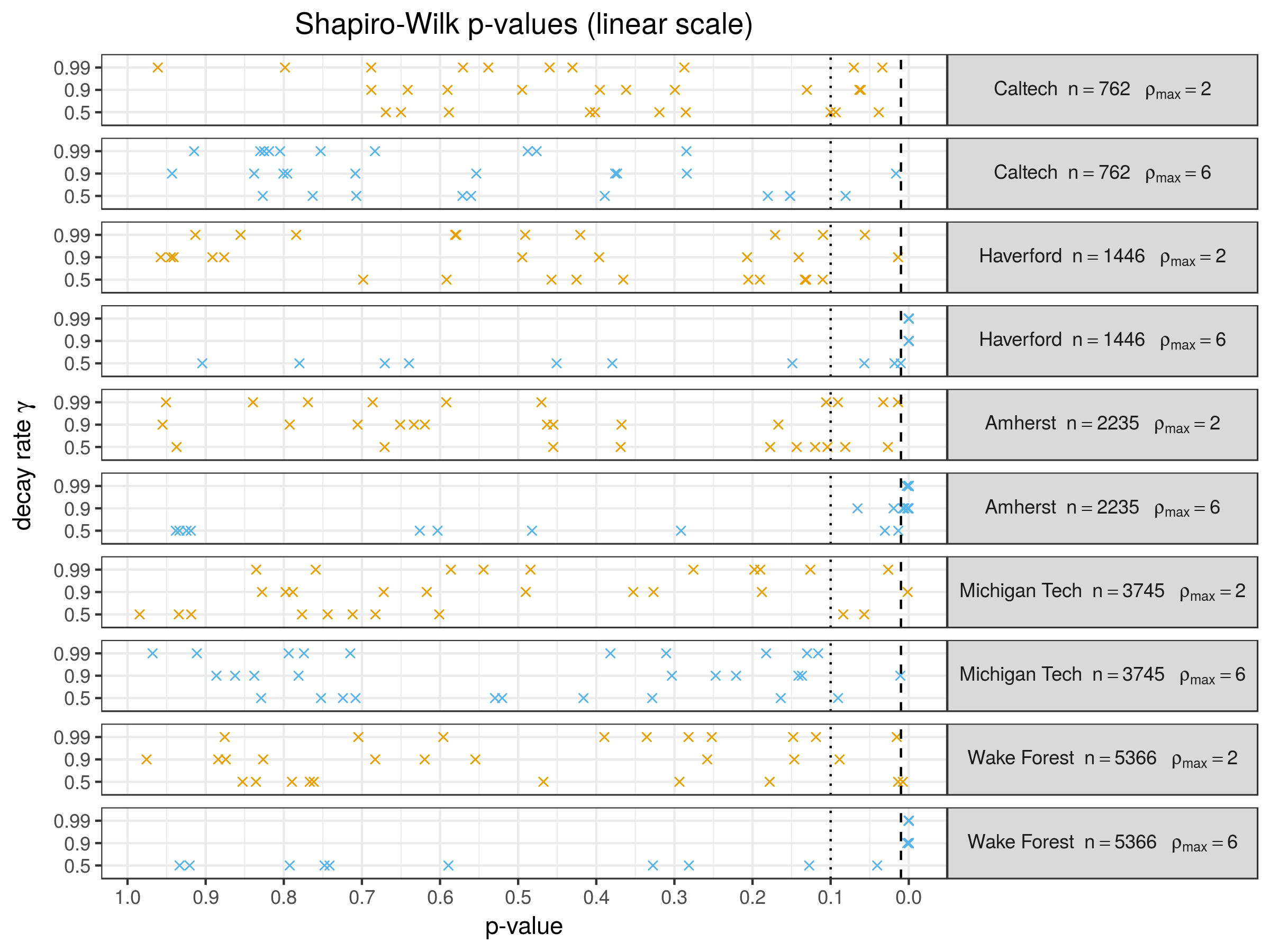}
\end{subfigure}
\begin{subfigure}[b]{0.77\textwidth}
   \includegraphics[width=\textwidth]{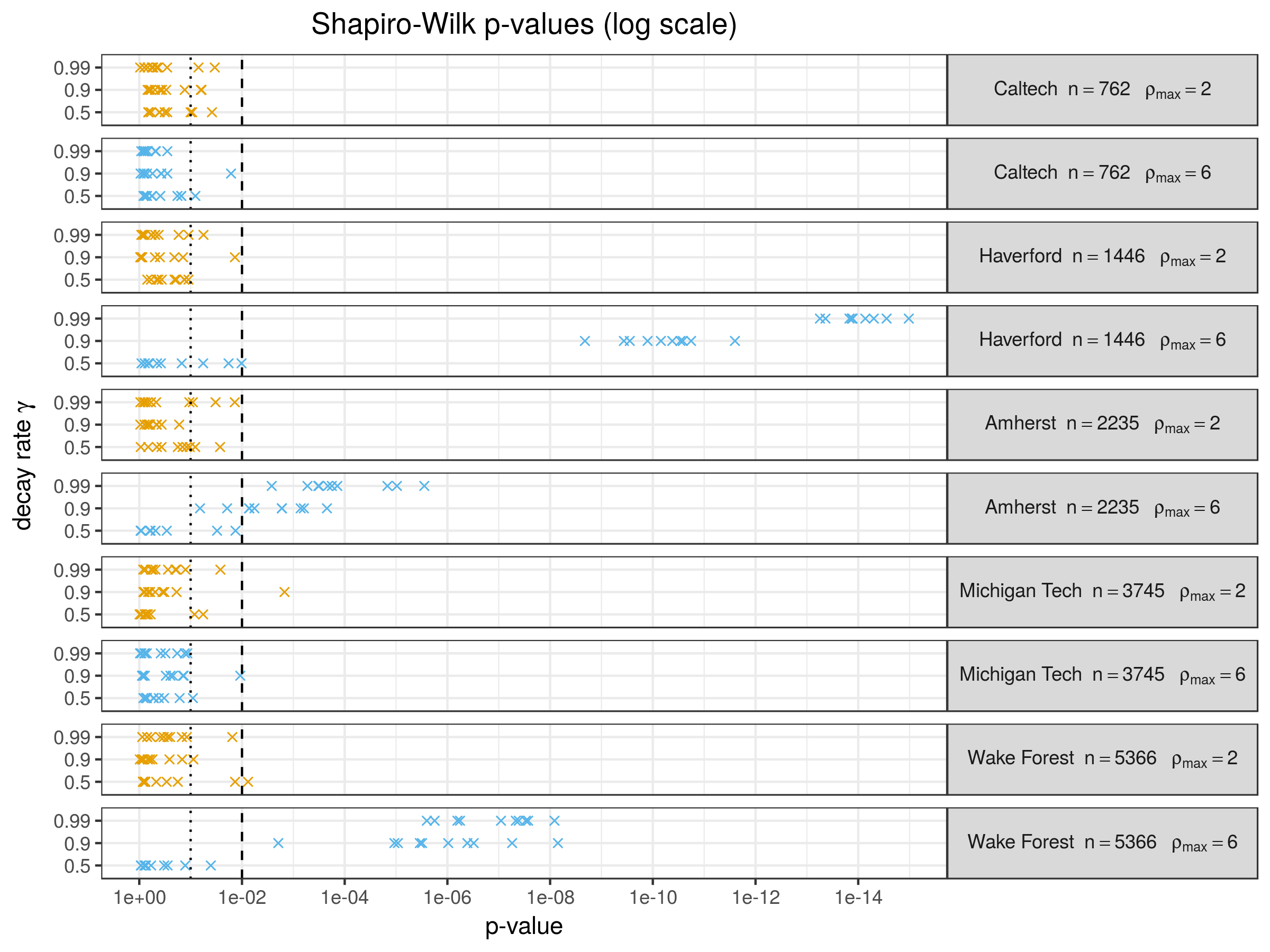}
\end{subfigure}
\caption{(top) Every data point is a $p$-value for the Shapiro-Wilk test against a Gaussian reference distribution.  Each panel represents a different network and dependency distance $\rhomax$ combination.  The panels with orange points correspond to $\rhomax = 2$ (less interference) and those with blue points correspond to $\rhomax = 6$ (more interference).  The vertical axis contains the three levels of the decay rate $\gamma$, ranging from $\gamma = 0.5$ (less interference) to $\gamma = 0.99$ (more interference).  The nominal cutoff values 0.1 (vertical dotted line) and 0.01 (vertical dashed line) are highlighted for reference.  (bottom) The same plot but using a logarithmic scale for the horizontal axis.}
\label{fig:sw}
\end{figure}

\subsection{Variance decompositions}
For this simulation we explore the relationship between the strength of interference and the resulting variance components.  We focus on the Caltech network, which, based on the previous simulation, appears to have a network structure such that the difference-in-means estimator for our response model appears to have a Gaussian distribution even under strong regimes of interference.  We draw a single set of $\alpha_i^{(0)}$ and $\alpha_i^{(1)}$ using the same distribution as the previous simulation, with exponential distributions of mean $1 / 0.3$ for the treatment group and mean $2$ for the control group.  We vary the maximum distance $\rhomax$ from $0$ to $5$ and the decay parameter $\gamma$ from $0.1$ to $0.9$ in increments of $0.1$.  For each parameter configuration, we draw 10,000 iterates of the treatment vector $\+W$ as iid Bernoulli$(0.5)$, and recompute the potential outcomes $Y_i^{(1)}$ and $Y_i^{(0)}$ each time.  We then use the potential outcomes to compute the variance components
\begin{align*}
\sigma_1^2 &= \avgin \E\left[(Y_i^{(1)} - \bar Y^{(1)})^2\right] \\
\sigma_0^2 &= \avgin \E\left[(Y_i^{(0)} - \bar Y^{(0)})^2\right] \\
\sigma_{01} &= \avgin \E\left[(Y_i^{(1)} - \bar Y^{(1)})(Y_i^{(0)} - \bar Y^{(0)})\right] \\
\sigma_\tau^2 &= n \var\left[\bar Y^{(1)} - \bar Y^{(0)}\right],
\end{align*}
where the expectation and variance are computed as finite population moments over the 10,000 simulation replicates.  We also compute the observed variance $\hat \sigma_{\text{DM}}^2$ of the difference-in-means estimator.  \hl{This is calculated as the empirical variance of the difference-in-means estimator over the 10,000 draws of the treatment vector.}

The SUTVA (conditional) variance is 
\[\sigma_{\text{SUTVA}}^2 = \sigma_1^2 + \sigma_0^2 + 2\sigma_{01}.\]
We display the ratio of the expected true variance of $\hat \tau$ to the conditional variance, $(\sigma_{\text{SUTVA}}^2 + \sigma_\tau^2) / \sigma_{\text{SUTVA}}^2$, as well as the observed ratio, $\hat \sigma_{\text{DM}}^2 / \sigma_{\text{SUTVA}}^2$.
The resulting ratios are displayed in Figure~\ref{fig:simulation-variances}.  The observed variances mostly track the expected variances.  Under SUTVA, $Y_i^{(0)}$ and $Y_i^{(1)}$ exhibit no additional variation so the observed variance appears to match $\sigma_\text{SUTVA}^2$.  As we allow units to influence units farther away in the graph, the variance ratio grows.  The discrepancy is not too large for fast decaying interference ($\gamma < 0.5$) but for $\gamma$ close to $1.0$ it can be drastic.  When $\rhomax = 5$ and $\gamma = 0.6$ the observed variance is only 7.4\% larger than $\sigma_\text{SUTVA}^2$, but for $\rhomax = 5$ and $\gamma = 0.9$ the observed discrepancy is 60.1\%.

\begin{figure}[t]
\centering
\includegraphics[width=0.7\textwidth]{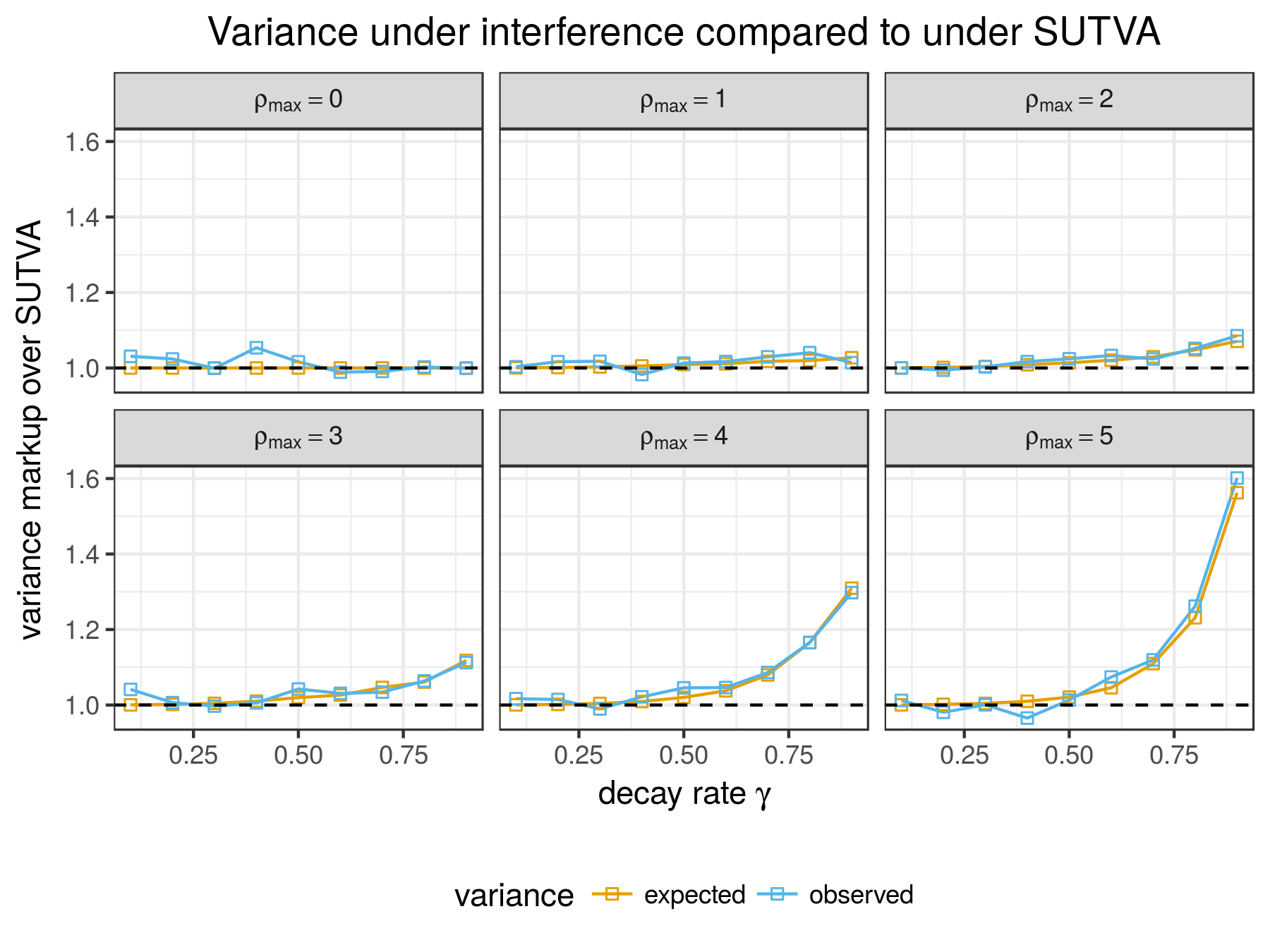}
\caption{Variance ratios for the Caltech network.  Each panel represents a different maximum distance $\rhomax$.  The horizontal axis is the decay rate $\gamma$ and the vertical axis marks the variance ratios.  The horizontal dotted line marks the baseline, which is a ratio of one.  The orange values are the expected variance ratios $(\sigma_{\text{SUTVA}}^2 + \sigma_\tau^2) / \sigma_{\text{SUTVA}}^2$, and the blue values are the observed variance ratios $\hat \sigma_{\text{DM}}^2 / \sigma_{\text{SUTVA}}^2$.  The upper-left most panel, $\rhomax = 0$, is the case when SUTVA is true.  The markup is greatest when $\rhomax$ and $\gamma$ are both large.  Note that the horizontal axis starts at 100\% and that the greatest observed ratio is about a 60\% increase in the variance over SUTVA.}
\label{fig:simulation-variances}
\end{figure}

\section{Discussion}
\label{sec:discussion}

In this work we have developed a framework for obtaining asymptotic results for causal estimators in randomized experiments in the presence of interference.  We contextualize the work of~\citet{savje2017average} within Stein's method and obtain asymptotic normality results.

Our two main results---one constraining the dependency degree (Theorem~\ref{thm:clt-depgraph}) and another placing conditions on how the treatment variables interact with the response variables (Theorem~\ref{thm:clt-perturbative})---highlight two general ways one may proceed for handling arbitrary interference.  The dependency graph approach follows a motif found in the interference literature of relying on local interference assumptions, such as the neighborhood treatment response condition.  Such assumptions, which enforce a sort of ``sparsity of interference,'' are often viewed as implausible yet necessary for tractability of results.  However, we have shown that progress is possible under certain dense regimes of interference.  \hl{Our result is still restrictive in the sense that it requires a condition of approximately sparse dependency.}  It is possible that more general advancements can be made by characterizing the behavior of the object $T$ in the main perturbative theorem (Lemma~\ref{lem:chatterjee}).

As discussed in Section~\ref{sec:setup}, the definition of $\sigma_\tau^2$ (equation~\eqref{eqn:sigma-tau} in Assumption~\ref{asm:variance-limits}) means that we have restricted ourselves to studying rate-optimal scenarios.  A useful extension would be to establish similar limiting results for the case when the difference-in-means $\bar Y^{(1)} - \bar Y^{(0)}$ converges at a slower-than-$\sqrt{n}$ rate.  A related issue is efficiency in the presence of interference.  The semiparametric efficiency bound \citep{hahn1998role} that serves as the basis for efficient estimation of average treatment effects in observational studies is reliant on independent units.  A characterization of similar semiparametric efficiency bounds for various levels of interference is a prerequisite for understanding whether efficient estimators remain optimal under interference.

\hl{We discuss how to approach statistical inference by noticing that $\sigma_\tau^2$ characterizes the difference between the variance under interference and under SUTVA.  It is also possible that estimates of $\sigma_\tau^2$ can serve as the basis for a statistical test of whether interference exists.  For a related recent idea, see~\citet{choi2018using}. }

Finally, our work is a novel application of Stein's method.  From a technical standpoint, our results demonstrate that tools from that literature can be used for establishing theoretical results for causal estimators under interference.  By overlaying the interference framework on top of Stein's method we are able to sidestep more complicated calculations or detailed assumptions about the structure of interference.  We have not addressed other statistical objects such as more general estimators and designs, or different estimands including the global treatment effect that compares all units in treatment to all units in control.  \hl{It is important to understand the behavior of the ``direct effect'' (EATE) estimand considered in \citet{savje2017average} and this paper, because it is the natural estimand of SUTVA-based estimators.  However it is worth considering the utility and interpretability of such an estimand in practice.}   Since interference at its core involves handling a dependent collection of random variables, we suspect that Stein's method may be useful for understanding other settings as well.

\bibliography{stein-clt}

\pagebreak
\appendix

\section{Proof of Theorem~\ref{thm:clt-depgraph}}
In this section we prove Theorem~\ref{thm:clt-depgraph}, restated here.
\thmdepgraph*
\begin{proof}
First, by conditioning on the $\sigma$-field defined in equation~\eqref{eqn:sigma-field},
\begin{align*}
n\var(\tilde \tau) &= n \E[\var(\tilde \tau | \cF)] + \var[\E(\tilde \tau | \cF)]
\end{align*}
The term $\var(\tilde \tau | \cF)$ is the variance of the Horvitz-Thompson estimator under SUTVA, which scales at rate $n^{-1/2}$.  It is written in terms of second moments of the outcomes $Y_i^{(1)}$ and $Y_i^{(0)}$, so the term $n \E[\var(\tilde \tau | \cF)]$ stabilizes by Assumption~\ref{asm:variance-limits}.  The second term is equal to $\var(\bar Y^{(1)} - \bar Y^{(0)}$ which also stabilizes by Assumption~\ref{asm:variance-limits}.  Therefore the entire variance  $n \var(\tilde \tau)$ converges to a non-zero limit $\sigma^2$.



Now decompose the Horvitz-Thompson estimator~\eqref{eqn:ht} as $\tilde \tau = \sumin \hat \tau_{i}$ where
\[\tilde \tau_i = \frac{1}{n} \left[\frac{W_iY_i^{(1)}}{\pi} - \frac{(1 - W_i)Y_i^{(0)}}{1 - \pi}\right].\]
Let $X_i = \sqrt{n}(\tilde \tau_i - \E[\tilde \tau_i])$ and denote $\sigma^2 = n\var(\tilde \tau)$.  Then 
\[S = S_n = \sumin X_i = \sqrt{n}(\tilde \tau - \tau) / \sigma.\]
By the uniform moment bound (Assumption~\ref{asm:bounded-moments}), $X_i = O_p(n^{-1/2})$, so for large enough $n$ there exist constants $C_1$ and $C_2$ such that
\[\bigg(\sumin \E[X_i^4]\bigg)^{1/2} \leq C_1 n^{-1/2}\]
and
\[\sumin \E|X_i|^3 \leq C_2 n^{-1/2}.\]
We can now apply Lemma~\ref{lem:dep-graph-bound}, which establishes for fixed $n$ that
\[\wass(S_n) \leq C_1 \frac{d_n^{3/2}}{n^{1/2}} + C_2 \frac{d_n^2}{n^{1/2}},\]
where we have ignored the $\sigma^2$ term because it stabilizes.
Therefore $\wass(S_n) \to 0$ whenever $d_n = o(n^{1/4})$, which is the constraint we have placed on the dependency graph (Assumption~\ref{asm:max-deg}).  Hence $S_n$ converges to a standard Gaussian random variable.
\end{proof}

\pagebreak
\section{Proof of Theorem~\ref{thm:clt-perturbative}}

\subsection{Proof of Corollary~\ref{cor:chatterjee}}
We first prove the version of Lemma~\ref{lem:chatterjee} for local dependencies.
\corchatterjee*
\begin{proof}
Notice that
\begin{align*}
\var(\E(T | S)) &\leq \var T \leq \frac{1}{4} \sum_{i, j = 1}^n \sum_{\substack{A \subset [n] \setminus \{i\} \\B \subset [n] \setminus \{j\} }} \frac{\cov(\Delta_if \Delta_i f^A, \Delta_jf \Delta_jf^B)}{n^2\binom{n}{|A|}\binom{n}{|B|}} \\
&\leq \frac{1}{4} \sum_{i, j = 1}^n \sum_{\substack{A \subset [n] \setminus \{i\} \\B \subset [n] \setminus \{j\} }} \frac{c_{i,j}}{n^2\binom{n}{|A|}\binom{n}{|B|}} = \frac{1}{4} \sum_{i, j = 1}^n c_{i,j}.
\end{align*}
Applying Lemma~\ref{lem:chatterjee} completes the proof.
\end{proof}

\subsection{Lemmas}
The lemmas in this section focus on getting a handle on the discrete derivative.  Throughout this section, $C_1, C_2, C_3, \dots$ indicate numerical constants that do not depend on $n$, and their values may change from line to line.

\begin{lemma}
\label{lem:elem-bound}
Let $A_i$ and $B_{i,r}$ be defined as in Lemma~\ref{lem:form-of-derivative} and assume the regularity conditions (Assumptions \ref{asm:design}-\ref{asm:variance-limits}).  For all $i$, $j$, $r$, and $s$, 
\begin{align}
|\cov(A_i, A_j)| \leq& \frac{C_1}{n^2} |\cov(Y_i, Y_j)| \label{eqn:elem-bound-1}\\
|\cov(B_{i,r}, A_j)| \leq& \left(\frac{C_1}{n^2} + \frac{C_2}{n^3}\right) |\cov(\Delta_i Y_r, Y_j)| + \frac{C_3}{n^3}|\cov(Y_r, Y_j)| \label{eqn:elem-bound-2}\\
|\cov(B_{i,r}, B_{j,s})| \leq& \left(\frac{C_1}{n^2} + \frac{C_2}{n^3}\right) |\cov(\Delta_i Y_r, \Delta_j Y_s)| + \left(\frac{C_3}{n^3} + \frac{C_4}{n^4}\right) |\cov(\Delta_i Y_r, Y_s)| \nonumber\\
&+ \left(\frac{C_5}{n^3} + \frac{C_6}{n^4}\right)|\cov(Y_r, \Delta_j Y_s)| + \frac{C_7}{n^4} |\cov(Y_r, Y_s)|, \label{eqn:elem-bound-3}
\end{align}
where the $C_k$ are constants, not necessarily the same from line to line.
\end{lemma}
\begin{proof}
Note that $B_{i,r}$ can be written as 
\[B_{i,r} = \left[\frac{W_r}{N_1}  - \frac{1 - W_r}{N_0}\right]\Delta_i Y_r + W_r Y_r^i \frac{W_i - W_i'}{N_1N_1'} - (1 - W_r)Y_r^i \frac{W_i - W_i'}{N_0N_0'}.\]
Equation~\eqref{eqn:elem-bound-1} follows from examining the form of $A_i$ and noting that $N_1 = O_p(n)$ and $N_0 = O_p(n)$.  For equation~\eqref{eqn:elem-bound-2}, note
\begin{align*}
|\cov(B_{i,r}, A_j)| &\leq \frac{C_1}{n^2} |\cov(\Delta_i Y_r, Y_j)| + \frac{C_2}{n^3}|\cov(Y_r^i, Y_j)|  \\
&= \frac{C_1}{n^2} |\cov(\Delta_i Y_r, Y_j)| + \frac{C_2}{n^3}(|\cov(Y_r, Y_j)| + |\cov(Y_r^i - Y_r, Y_j)|),
\end{align*}
which gives equation~\eqref{eqn:elem-bound-2}.
 Similarly, we have
\begin{align*}
|\cov(B_{i,r}, B_{j,s})| \leq& \frac{C_1}{n^2} |\cov(\Delta_i Y_r, \Delta_j Y_s)| + \frac{C_2}{n^3} |\cov(\Delta_i Y_r, Y_s^j)| \\
&+ \frac{C_3}{n^3} |\cov(Y_r^i, \Delta_j Y_s)| +  \frac{C_4}{n^4} |\cov(Y_r^i, Y_s^j)|,
\end{align*}
and rewriting $Y_r^i = Y_r - \Delta_i Y_r$ and $Y_s^j = Y_s - \Delta_i Y_s$ gives equation~\eqref{eqn:elem-bound-3}.
\end{proof}

\begin{lemma}
\label{lem:deriv-bound}
Under the regularity conditions (Assumptions \ref{asm:design}-\ref{asm:variance-limits}), there exist constants $C_1$ through $C_5$ such that
\begin{align*}
\frac{1}{n}\sum_{i,j} |\cov(\Delta_i f_n, \Delta_j f_n)| \leq& \frac{C_1}{n^2} \sum_{i,j} |\cov(Y_i, Y_j)| + \left(\frac{C_2}{n^2} + \frac{C_3}{n^3}\right) \sum_{i,j}\sum_{r \neq i} |\cov(\Delta_i Y_r, Y_j)| \\
&+ \left(\frac{C_4}{n^2} + \frac{C_5}{n^3}\right) \sum_{i,j} \sum_{r \neq i} \sum_{s \neq j} |\cov(\Delta_i Y_r, \Delta_j Y_s)|.
\end{align*}
\end{lemma}
\begin{proof}
By expanding the form of the discrete derivative~\eqref{eqn:discrete-deriv}, we have
\begin{align*}
\frac{1}{n}|\cov(\Delta_i f_n, \Delta_j f_n)| =& |\cov(A_i, A_j)| + \sum_{r \neq i} |\cov(B_{i,r}, A_j)| \\
&+ \sum_{s \neq j} |\cov(A_i, B_{j,s})| + \sum_{r \neq i} \sum_{s \neq j} |\cov(B_{i,r}, B_{j,s})|.
\end{align*}
By summing over $i$ and $j$ substituting the bounds from Lemma~\ref{lem:elem-bound}, the right-hand side above is bounded above by
\begin{align*}
&\frac{1}{n^2} \sum_{i,j} \Bigg[C_1|\cov(Y_i, Y_j)| + \sum_{r\neq i} \left[\left(C_2 + \frac{C_3}{n}\right)|\cov(\Delta_i Y_r, Y_j)| + \frac{C_4}{n} |\cov(Y_r, Y_j)|\right] \\
&+ \sum_{s\neq j} \left[\left(C_5 + \frac{C_6}{n}\right)|\cov(Y_i, \Delta_j Y_s)| + \frac{C_7}{n} |\cov(Y_i, Y_s)|\right] \\
&+ \sum_{r \neq i} \sum_{s \neq j} \bigg[\left(C_8 + \frac{C_9}{n}\right)|\cov(\Delta_i Y_r, \Delta_j Y_s)| + \left(\frac{C_{10}}{n} + \frac{C_{11}}{n^2}\right) |\cov(\Delta_i Y_r, Y_s)| \\
&+ \left(\frac{C_{12}}{n} + \frac{C_{13}}{n^2}\right) |\cov(Y_r, \Delta_j Y_s)| + \frac{C_{14}}{n^2}|\cov(Y_r, Y_s)|\bigg]\Bigg] .
\end{align*}
We now exploit the symmetry in the summations and combine terms to give the desired result.
\end{proof}

\begin{lemma}[Restricted interference]
Under the regularity conditions (Assumptions~\ref{asm:design}-\ref{asm:variance-limits} and~\ref{asm:approx-regularity}) and approximate local interference (Assumption~\ref{asm:approx-local}), for every unit $i$, the total amount of interference that results from perturbing the treatment of unit $i$ satisfies
\[\sum_{r \neq i} |\Delta_i Y_r| = O_p(n^{1/3}).\]
\label{lem:restricted-interference}
\end{lemma}
\begin{proof}
We divide the interference emanating from unit $i$ into collections of ``weak'' and ``strong'' interference, this partition being specified by the neighborhood $\cN_i^{H_n}$.
\begin{align*}
\sum_{r \neq i} |\Delta_i Y_r| &= \sum_{r \notin \cN_i^{H_n}} |\Delta_i Y_r| + \sum_{r \in \cN_i^{H_n}} |\Delta_i Y_r|.
\end{align*}
The first summand tends to $0$ as $n \to \infty$ by Assumption~\ref{asm:approx-local}.   The second summand is bounded above by $Ch_n$ because of the uniform moment bound (Assumption~\ref{asm:bounded-moments}), and the result follows since $h_n = o(n^{1/3})$.
Here $C$ is a constant, and the quantities $h_n$, $\delta_k$, $H_n$, and $\cN_i^{H_n}$ are as defined in Assumption~\ref{asm:approx-local}.
\end{proof}

\begin{lemma}
Under the regularity conditions (Assumptions~\ref{asm:design}-\ref{asm:variance-limits} and~\ref{asm:approx-regularity}) and approximate local interference (Assumption~\ref{asm:approx-local}),
\label{lem:useful1}
\[\frac{1}{n^2} \sumin \sum_{r \neq i} |\cov(\Delta_i Y_r, Y_i)| = o(1).\]
\end{lemma}
\begin{proof}
By the uniform moment bound there exists a constant $C$ such that
\[\frac{1}{n^2} \sumin \sum_{r \neq i} |\cov(\Delta_i Y_r, Y_i)| \leq \frac{C}{n^2} \sumin \sum_{r \neq i} \E|\Delta_i Y_r|.\]
Now using $\sum_{r \neq i} |\Delta_i Y_r| = O_p(n^{1/3})$, as established by Lemma~\ref{lem:restricted-interference}, we find that the right hand side of the inequality above is bounded above by
\[\frac{C}{n^2} \sum_{i=1}^n n^{1/3} = Cn^{-2/3} = o(1).\]
\end{proof}

\begin{lemma}
Under the regularity conditions (Assumptions~\ref{asm:design}-\ref{asm:variance-limits} and~\ref{asm:approx-regularity}) and approximate local interference (Assumption~\ref{asm:approx-local}),
\label{lem:useful2}
\begin{align*}
\frac{1}{n^2} \sum_{i,j} \sum_{\substack{r \neq i \\ r \neq j}} |\cov(\Delta_i Y_r, \Delta_j Y_r)| = o(1).
\end{align*}
\end{lemma}
\begin{proof}
Denote $\Delta_r^{i,j} = \cov(\Delta_i Y_r, \Delta_j Y_r)$.  We proceed by partitioning the sum depending on whether $r$ belongs to the neighborhoods of $i$ and $j$ as defined in Assumption~\ref{asm:approx-local}.  That is, we can write
\begin{align*}
\frac{1}{n^2} \sum_{i,j} \sum_{\substack{r \neq i \\ r \neq j}} \Delta_r^{i,j} &= \frac{1}{n^2} \sum_{i,j} \left[
\sum_{\substack{r \in \cN_i^{H_n} \\ r \in \cN_j^{H_n}}}|\Delta_r^{i,j}| + 
\sum_{\substack{r \in \cN_i^{H_n} \\ r \not \in \cN_j^{H_n}}}|\Delta_r^{i,j}| + 
\sum_{\substack{r \not \in \cN_i^{H_n} \\ r \in \cN_j^{H_n}}}|\Delta_r^{i,j}| + 
\sum_{\substack{r \not \in \cN_i^{H_n} \\ r \not \in \cN_j^{H_n}}}|\Delta_r^{i,j}|\right] \\
&\leq \frac{1}{n^2} \sum_{i,j} \left[
\sum_{\substack{r \in \cN_i^{H_n} \\ r \in \cN_j^{H_n}}}|\Delta_r^{i,j}| + 
\sum_{r \not \in \cN_j^{H_n}}|\Delta_r^{i,j}| + 
\sum_{r \not \in \cN_i^{H_n}}|\Delta_r^{i,j}| + 
\sum_{r \not \in \cN_i^{H_n}}|\Delta_r^{i,j}|\right] 
\end{align*}
Now, by Assumption~\ref{asm:approx-local}, each of the inner sums of the last three terms tends to zero in the limit (and the outer sums also tend to zero because there are $n^2$ summands offset by the $n^2$ in the denominator).  For the first term, the sum is zero whenever the intersection of $\cN_i^{H_n}$ and $\cN_j^{H_n}$ is empty, and of order $h_n$ otherwise.  Therefore, 
\begin{align*}
&\frac{1}{n^2} \sum_{i,j} \sum_{\substack{r \neq i \\ r \neq j}} \Delta_r^{i,j}\leq \frac{C}{n^2} \sum_{i,j} h_n\one(|\cN_i^{H_n} \cap \cN_j^{H_n}| > 0) + o(1) \\
& \leq \frac{Ch_n}{n^2} \sumin \sum_{s=1}^n \sum_{j=1}^n \one(s \in \cN_i^{H_n}, j \in \cN_{s}^{H_n}) + o(1)\\
&\leq \frac{Ch_n^3}{n} + o(1).
\end{align*}
The proof is finished by noting that $h_n = o(n^{1/3})$, as specified by Assumption~\ref{asm:approx-local}.
\end{proof}

\begin{lemma}
\label{lem:delta-bounds}
In addition to the regularity conditions (Assumptions \ref{asm:design}-\ref{asm:variance-limits}), assume that Assumption~\ref{asm:approx-local} (approximate local independence) holds.  Then for all $i \in [n]$ and $A \in [n] \setminus \{i\}$,
\begin{align}
|\Delta_i f_n| &= O_p(n^{-1/2}) \label{eqn:delta-bound} \\
|\Delta_i f_n^A| &= O_p(n^{-1/2}). \label{eqn:delta-A-bound}
\end{align}
\end{lemma}
\begin{proof}
By a similar argument as in Lemma~\ref{lem:deriv-bound}, 
\begin{align*}
\E(\Delta_i f_n)^2 &\leq n \left[\var(A_i) + \sum_{r \neq i} \cov(A_i, B_{i,r}) + \sum_{r \neq i} \sum_{s \neq i} \cov(B_{i,r}, B_{i,s})\right] \\
&\leq \frac{C_1}{n} \var(Y_i) + \left(\frac{C_2}{n} + \frac{C_3}{n^2}\right) \sum_{r \neq i} \cov(Y_i, \Delta_i Y_r) + \left(\frac{C_4}{n} + \frac{C_5}{n^2}\right)\sum_{r \neq i} \sum_{s \neq i} \cov(\Delta_i Y_r, \Delta_i Y_s) \\
&\leq \frac{C_1}{n} + \left(\frac{C_2}{n} + \frac{C_3}{n^2}\right) \sum_{r \neq i} \E|\Delta_i Y_r| + \left(\frac{C_4}{n} + \frac{C_5}{n^2}\right)\sum_{r \neq i} \sum_{s \neq i} \E[|\Delta_i Y_r\Delta_i Y_s|] \\
&= \frac{C_1}{n} + \left(\frac{C_2}{n} + \frac{C_3}{n^2}\right) \sum_{r \neq i} \E|\Delta_i Y_r| + \left(\frac{C_4}{n} + \frac{C_5}{n^2}\right)\bigg(\sum_{r \neq i} \E|\Delta_i Y_r|\bigg)^2.
\end{align*}
The whole right-hand side is then $O(n^{-1})$ by the fact that $\sum_{r \neq i} |\Delta_i Y_r| = O_p(1)$ (Assumption~\ref{asm:approx-local}).
Then~\eqref{eqn:delta-bound} follows from Markov's inequality.  Equation~\eqref{eqn:delta-A-bound} immediately follows because $\Delta_i f_n^A$ is equal in distribution to $\Delta_i f_n$.
\end{proof}

\subsection{Proof of main theorem}
We are now ready to prove Theorem~\ref{thm:clt-perturbative}, restated here.
\thmperturbative*
\begin{proof}
We first compute the limiting variance $\sigma^2 := \limn n \var(\hat \tau)$.  
Let $\cF$ be the $\sigma$-field defined by equation~\eqref{eqn:sigma-field}.  By conditioning on $\cF$ we have
\[\var(\hat \tau) = \E \left[\var\left[\hat \tau\big | \cF\right]\right] + \var \left[\E\left[\hat \tau \big | \cF \right]\right].\]
Now,
\[\var[\hat \tau | \cF] = \var\left[\sumin \frac{W_iY_i^{(1)}}{N_1} - \frac{(1 - W_i)Y_i^{(0)}}{N_0}\bigg| Y_i^{(0)}, Y_i^{(1)}\right]\]
is the usual variance of a difference-in-means estimator under SUTVA, i.e.\ fixed potential outcomes.  This is known to be~\citep[see for example][]{lin2013agnostic}
\[\limn n \E[\var[\hat \tau | \cF]] = \frac{1 - \pi}{\pi} \sigma_1^2 + \frac{\pi}{1 - \pi} \sigma_0^2 + 2\sigma_{01}.\]
For the second term, we have $\E[\hat \tau | \cF] = \bar Y_n^{(1)} - \bar Y_n^{(0)}$, so 
\[\limn n \var[\E[\hat \tau | \cF]] = \sigma_\tau^2\]
by Assumption~\ref{asm:variance-limits}.  This produces the variance expression~\eqref{eqn:asymptotic-variance}.

Since the variance term $\sigma^2$ of expression \eqref{eqn:wass-perturbative} in Corollary~\ref{cor:chatterjee} stabilizes, it is sufficient to show
\[\limn \bigg(\sum_{i,j} c_{i,j}\bigg)^{1/2} = 0 \qquad \text{and} \qquad
\limn \sumin \E|\Delta_i f_n|^3 = 0.\]
Since $|\Delta_i f_n^A| = O_p(n^{-1/2})$ by equation~\eqref{eqn:delta-A-bound} of Lemma~\ref{lem:delta-bounds}, there exists a constant $C$ such that
\[|\cov(\Delta_i f_n\Delta_i f_n^A, \Delta_j f_n \Delta_jf_n^B)| \leq \frac{C}{n} |\cov(\Delta_i f_n, \Delta_j f_n)|.\]
Then by Lemma~\ref{lem:deriv-bound}, there exist constants $c_{i,j} \geq 0$ such that
\[|\cov(\Delta_i f_n\Delta_i f_n^A, \Delta_j f_n \Delta_jf_n^B)| \leq c_{i,j}\]
and
\begin{align*}
\sum_{i,j} c_{i,j} \leq & \frac{C_1}{n^2} \sum_{i,j}|\cov(Y_i, Y_j)| + \frac{C_2}{n^2}\bigg[ \sumin \sum_{r \neq i} |\cov(\Delta_i Y_r, Y_i)| + \sumin \sum_{j \neq i} \sum_{r \neq i}|\cov(\Delta_i Y_r, Y_j)|\bigg]\\
& + \frac{C_3}{n^2} \bigg[\sum_{i,j} \sum_{\substack{r \neq i \\ r \neq j}} |\cov(\Delta_i Y_r, \Delta_j Y_r)| + \sum_{i,j} \sum_{r \neq i} \sum_{\substack{s \neq j \\ s \neq r}} |\cov(\Delta_i Y_r, \Delta_j Y_s)| \bigg].
\end{align*}
Each of the five terms in the bound captures a different relationship among the responses and discrete derivatives.  The first term measures a global covariance structure which tends to zero by Assumption~\ref{asm:approx-regularity}.  The third and fifth terms concern covariances among distinct actors, which are also negligible by Assumption~\ref{asm:approx-regularity}.  The second and fourth terms are the only ones that include elements measuring strong interference.  These two terms are handled by Lemmas~\ref{lem:useful1} and \ref{lem:useful2}, respectively.  So we conclude
\[\limn \bigg(\sum_{i,j} c_{i,j}\bigg)^{1/2} = 0.\]
Finally, by equation~\eqref{eqn:delta-bound} of Lemma~\ref{lem:delta-bounds}, $\E|\Delta_i f_n|^3 = O(n^{-3/2})$.  Hence
\[\sumin \E|\Delta_i f_n|^3 = O(n^{-1/2})\]
and so tends to zero.
\end{proof}

\pagebreak
\section{Proof of Proposition~\ref{prop:variance-estimation}}
\begin{proposition}
\label{prop:variance-estimation}
Under regularity conditions (Assumptions \ref{asm:design}-\ref{asm:variance-limits}) and restricted interference (either Assumption~\ref{asm:max-deg} or Assumptions~\ref{asm:approx-regularity}-\ref{asm:approx-local}), $\hat \sigma_\tau^2$ is consistent for $\sigma_\tau^2$. 
\end{proposition}
\begin{proof}
We wish to show that
\[\hat V_\tau^2 = \bar Y_1^2 + \bar Y_0^2 - 2\bar Y_1 \bar Y_0 - \hat \tau^2\]
is consistent for $\var(T) = \E[T^2] - \tau^2$.  It is already established that $\hat \tau \pto \tau$, so it suffices to show that $\var(\bar Y_1^2) \to 0$ (and $\var(\bar Y_0^2) \to 0$ is similar).

Now, the variance is decomposed as
\[\var(\bar Y_1^2) =  \E(\var (\bar Y_1^2 | \cF)) + \var(\E(\bar Y_1^2 | \cF))\]
where $\cF$ is the $\sigma$-field defined by equation \eqref{eqn:sigma-field} representing ``conditioning on SUTVA.''  Since $\bar Y_1^2$ is consistent under SUTVA, we have $\var(\bar Y_1^2 | \cF) \to 0$.  Hence the first term is zero.  For the second term, $\E(\bar Y_1^2 | \cF) = (\bar Y^{(1)})^2$, and so we require that $\var((\bar Y^{(1)})^2) \to 0$.  Notice that if $Y_i$ have maximal dependency degree $o(n^{k})$ then $Y_i^2$ have maximal dependency degree $o(n^{2k})$.  Therefore consistency for $(\bar Y^{(1)})^2$ follows from Proposition 2 of~\citet{savje2017average} whenever $k < 1/2$. (see also Assumption 2 of that paper).  Hence this is satisfied for the (approximate) dependency degree restrictions used in this paper, where $k = 1/4$ or $k = 1/3$.  

\end{proof}

\pagebreak
\section{Tables of simulation results}


\begin{table}[ht]
\centering
\begin{tabular}{cc|cc|c|ccc}
  \toprule
  \multicolumn{2}{c|}{Network} & \multicolumn{2}{c|}{Parameters} & \multicolumn{1}{c}{SW statistic} & \multicolumn{3}{|c}{SW $p$-values} \\
school & nodes & $\rhomax$ & $\gamma$ & average & avg & min & max \\ 
  \midrule
  Caltech &  762 &    2 & 0.5 & 0.996 & 0.355 & 0.038 & 0.669 \\ 
  Caltech &  762 &    2 & 0.9 & 0.996 & 0.373 & 0.061 & 0.688 \\ 
  Caltech &  762 &    2 & 0.99 & 0.997 & 0.484 & 0.034 & 0.961 \\ 
  Caltech &  762 &    6 & 0.5 & 0.997 & 0.438 & 0.081 & 0.827 \\ 
  Caltech &  762 &    6 & 0.9 & 0.997 & 0.569 & 0.016 & 0.943 \\ 
  Caltech &  762 &    6 & 0.99 & 0.998 & 0.688 & 0.285 & 0.915 \\ 
  Haverford & 1446 &    2 & 0.5 & 0.996 & 0.331 & 0.110 & 0.698 \\ 
  Haverford & 1446 &    2 & 0.9 & 0.997 & 0.586 & 0.014 & 0.958 \\ 
  Haverford & 1446 &    2 & 0.99 & 0.997 & 0.496 & 0.056 & 0.913 \\ 
  Haverford & 1446 &    6 & 0.5 & 0.996 & 0.406 & 0.010 & 0.904 \\ 
  Haverford & 1446 &    6 & 0.9 & 0.957 & 0.000 & 0.000 & 0.000 \\ 
  Haverford & 1446 &    6 & 0.99 & 0.928 & 0.000 & 0.000 & 0.000 \\ 
Amherst & 2235 &    2 & 0.5 & 0.996 & 0.309 & 0.027 & 0.937 \\ 
  Amherst & 2235 &    2 & 0.9 & 0.997 & 0.581 & 0.167 & 0.955 \\ 
  Amherst & 2235 &    2 & 0.99 & 0.996 & 0.455 & 0.014 & 0.951 \\ 
  Amherst & 2235 &    6 & 0.5 & 0.997 & 0.576 & 0.013 & 0.938 \\ 
  Amherst & 2235 &    6 & 0.9 & 0.991 & 0.011 & 0.000 & 0.066 \\ 
  Amherst & 2235 &    6 & 0.99 & 0.986 & 0.000 & 0.000 & 0.003 \\ 
  Michigan Tech & 3745 &    2 & 0.5 & 0.997 & 0.649 & 0.057 & 0.985 \\ 
  Michigan Tech & 3745 &    2 & 0.9 & 0.996 & 0.506 & 0.001 & 0.828 \\ 
  Michigan Tech & 3745 &    2 & 0.99 & 0.996 & 0.403 & 0.026 & 0.835 \\ 
  Michigan Tech & 3745 &    6 & 0.5 & 0.997 & 0.506 & 0.091 & 0.829 \\ 
  Michigan Tech & 3745 &    6 & 0.9 & 0.996 & 0.443 & 0.011 & 0.886 \\ 
  Michigan Tech & 3745 &    6 & 0.99 & 0.997 & 0.528 & 0.116 & 0.968 \\ 
  Wake Forest & 5366 &    2 & 0.5 & 0.996 & 0.497 & 0.008 & 0.853 \\ 
  Wake Forest & 5366 &    2 & 0.9 & 0.997 & 0.591 & 0.089 & 0.976 \\ 
  Wake Forest & 5366 &    2 & 0.99 & 0.996 & 0.372 & 0.015 & 0.876 \\ 
  Wake Forest & 5366 &    6 & 0.5 & 0.997 & 0.550 & 0.040 & 0.933 \\ 
  Wake Forest & 5366 &    6 & 0.9 & 0.979 & 0.000 & 0.000 & 0.002 \\ 
  Wake Forest & 5366 &    6 & 0.99 & 0.975 & 0.000 & 0.000 & 0.000 \\ 
   \bottomrule
\end{tabular}
\caption{Summary of Shapiro-Wilk $p$-values from Simulation 1.  Average, minimum, and maximum are taken over the 10 instances of the response.}
\label{table:simulation1}
\end{table}

\pagebreak
\begin{table}[H]
\small
\centering
\begin{tabular}{rr|rrr|rr}
  \toprule
  \multicolumn{2}{c|}{Parameters} & \multicolumn{3}{c}{Variances} & \multicolumn{2}{|c}{Ratios to SUTVA} \\
$\rhomax$ & $\gamma$ & SUTVA & expected & observed & expected & observed \\ 
  \midrule
   0 & 0.1 & 14.770 & 14.770 & 15.228 & 1.000 & 1.031 \\ 
     0 & 0.2 & 15.205 & 15.205 & 15.570 & 1.000 & 1.024 \\ 
     0 & 0.3 & 14.690 & 14.690 & 14.680 & 1.000 & 0.999 \\ 
     0 & 0.4 & 15.382 & 15.382 & 16.208 & 1.000 & 1.054 \\ 
     0 & 0.5 & 14.478 & 14.478 & 14.714 & 1.000 & 1.016 \\ 
     0 & 0.6 & 14.321 & 14.321 & 14.164 & 1.000 & 0.989 \\ 
     0 & 0.7 & 16.674 & 16.674 & 16.521 & 1.000 & 0.991 \\ 
     0 & 0.8 & 17.574 & 17.574 & 17.623 & 1.000 & 1.003 \\ 
     0 & 0.9 & 16.453 & 16.453 & 16.440 & 1.000 & 0.999 \\ 
     1 & 0.1 & 12.717 & 12.722 & 12.758 & 1.000 & 1.003 \\ 
     1 & 0.2 & 14.845 & 14.864 & 15.094 & 1.001 & 1.017 \\ 
     1 & 0.3 & 14.694 & 14.736 & 14.954 & 1.003 & 1.018 \\ 
     1 & 0.4 & 14.282 & 14.360 & 14.034 & 1.005 & 0.983 \\ 
     1 & 0.5 & 12.739 & 12.856 & 12.906 & 1.009 & 1.013 \\ 
     1 & 0.6 & 16.073 & 16.251 & 16.346 & 1.011 & 1.017 \\ 
     1 & 0.7 & 13.262 & 13.497 & 13.655 & 1.018 & 1.030 \\ 
     1 & 0.8 & 15.247 & 15.546 & 15.867 & 1.020 & 1.041 \\ 
     1 & 0.9 & 14.324 & 14.713 & 14.528 & 1.027 & 1.014 \\ 
     2 & 0.1 & 16.199 & 16.205 & 16.198 & 1.000 & 1.000 \\ 
     2 & 0.2 & 17.088 & 17.113 & 16.990 & 1.001 & 0.994 \\ 
     2 & 0.3 & 15.325 & 15.386 & 15.373 & 1.004 & 1.003 \\ 
     2 & 0.4 & 14.046 & 14.168 & 14.283 & 1.009 & 1.017 \\ 
     2 & 0.5 & 15.489 & 15.703 & 15.868 & 1.014 & 1.024 \\ 
     2 & 0.6 & 16.697 & 17.040 & 17.247 & 1.021 & 1.033 \\ 
     2 & 0.7 & 17.665 & 18.186 & 18.088 & 1.029 & 1.024 \\ 
     2 & 0.8 & 15.419 & 16.159 & 16.219 & 1.048 & 1.052 \\ 
     2 & 0.9 & 14.598 & 15.631 & 15.846 & 1.071 & 1.085 \\ 
     3 & 0.1 & 14.095 & 14.100 & 14.678 & 1.000 & 1.041 \\ 
     3 & 0.2 & 14.267 & 14.292 & 14.359 & 1.002 & 1.006 \\ 
     3 & 0.3 & 14.798 & 14.863 & 14.755 & 1.004 & 0.997 \\ 
     3 & 0.4 & 13.442 & 13.581 & 13.517 & 1.010 & 1.006 \\ 
     3 & 0.5 & 12.762 & 13.013 & 13.302 & 1.020 & 1.042 \\ 
     3 & 0.6 & 16.095 & 16.519 & 16.592 & 1.026 & 1.031 \\ 
     3 & 0.7 & 14.900 & 15.595 & 15.410 & 1.047 & 1.034 \\ 
     3 & 0.8 & 18.009 & 19.103 & 19.158 & 1.061 & 1.064 \\ 
     3 & 0.9 & 14.031 & 15.690 & 15.607 & 1.118 & 1.112 \\ 
     4 & 0.1 & 12.765 & 12.771 & 12.980 & 1.000 & 1.017 \\ 
     4 & 0.2 & 13.902 & 13.927 & 14.105 & 1.002 & 1.015 \\ 
     4 & 0.3 & 15.799 & 15.866 & 15.638 & 1.004 & 0.990 \\ 
     4 & 0.4 & 15.210 & 15.352 & 15.541 & 1.009 & 1.022 \\ 
     4 & 0.5 & 14.311 & 14.601 & 14.962 & 1.020 & 1.046 \\ 
     4 & 0.6 & 16.144 & 16.755 & 16.893 & 1.038 & 1.046 \\ 
     4 & 0.7 & 15.692 & 16.942 & 17.043 & 1.080 & 1.086 \\ 
     4 & 0.8 & 16.461 & 19.185 & 19.188 & 1.165 & 1.166 \\ 
     4 & 0.9 & 18.759 & 24.566 & 24.343 & 1.310 & 1.298 \\ 
     5 & 0.1 & 14.747 & 14.752 & 14.928 & 1.000 & 1.012 \\ 
     5 & 0.2 & 13.261 & 13.286 & 13.006 & 1.002 & 0.981 \\ 
     5 & 0.3 & 15.400 & 15.467 & 15.409 & 1.004 & 1.001 \\ 
     5 & 0.4 & 14.980 & 15.127 & 14.462 & 1.010 & 0.965 \\ 
     5 & 0.5 & 14.784 & 15.094 & 14.978 & 1.021 & 1.013 \\ 
     5 & 0.6 & 14.554 & 15.221 & 15.635 & 1.046 & 1.074 \\ 
     5 & 0.7 & 14.374 & 15.951 & 16.093 & 1.110 & 1.120 \\ 
     5 & 0.8 & 16.307 & 20.078 & 20.575 & 1.231 & 1.262 \\ 
     5 & 0.9 & 15.546 & 24.286 & 24.894 & 1.562 & 1.601 \\ 
   \bottomrule
\end{tabular}
\caption{Table of variances for the Caltech network from Simulation 2.}
\label{table:simulation2}
\end{table}

\end{document}